\documentclass[11pt,a4paper]{article}

\usepackage{amssymb,amsbsy,amsmath,amsfonts,amssymb,amscd,amsthm}
\usepackage[T1]{fontenc}
\usepackage[utf8]{inputenc}
\usepackage{graphicx}
\usepackage[all]{xy}
\usepackage{pdfpages}
\usepackage[linktocpage]{hyperref}
\usepackage{latexsym, amssymb, amsthm}
\usepackage{dsfont}
\usepackage{color}
\usepackage{enumerate}
\usepackage{geometry}[a4]
\usepackage{import}

\newcommand{\G}{\mathcal{G}}
\newcommand{\SSS}{\mathcal{S}}
\DeclareMathAlphabet\gothic{U}{euf}{m}{n}

\newtheorem{theorem}{Theorem}
\newtheorem{lemma}[theorem]{Lemma}
\newtheorem{proposition}[theorem]{Proposition}
\newtheorem{corollary}[theorem]{Corollary}

\newtheorem{remark}[theorem]{Remark}

\newcommand{\LF}{\overrightarrow{\Delta}}
\newcommand{\dtent}{\frac{\mathrm{d}t \mathrm{d}y}{Vol(y,t^{1/2})}}
\newcommand{\M}{\mathcal{M}}

\title{Conical square functionals on Riemannian manifolds}
\author{Thomas Cometx - Institut de Mathématiques de Bordeaux}
\date{}
\begin{document}
\maketitle
\noindent
{\bf Abstract:} Let $L = \Delta + V$ be Schrödinger operator with a non-negative potential $V$ on a complete Riemannian manifold $M$. We prove that the conical square functional associated with $L$ is bounded on $L^p$ under different assumptions. This functional is defined by 
$$ \G_L (f) (x) = \left( \int_0^\infty \int_{B(x,t^{1/2})} |\nabla e^{-tL} f(y)|^2 + V |e^{-tL} f(y)|^2 \frac{\mathrm{d}t \mathrm{d}y}{Vol(y,t^{1/2})} \right)^{1/2}.$$
For $p \in [2,+\infty)$ we show that it is sufficient to assume that the manifold has the volume doubling property whereas for $p \in (1,2)$ we need extra assumptions of $L^p-L^2$ of diagonal estimates for $\{ \sqrt{t} \nabla e^{-tL}, t\geq 0 \}$ and $ \{ \sqrt{t} \sqrt{V} e^{-tL} , t \geq 0\}$.
Given a bounded holomorphic function $F$ on some angular sector, we introduce the generalized conical vertical square functional
$$\G_L^F (f) (x) = \left( \int_0^\infty \int_{B(x,t^{1/2})} |\nabla F(tL) f(y)|^2 + V |F(tL) f(y)|^2 \frac{\mathrm{d}t \mathrm{d}y}{Vol(y,t^{1/2})} \right)^{1/2}$$ and prove its boundedness on $L^p$ if $F$ has sufficient decay at zero and infinity. 
We also consider conical square functions associated with the Poisson semigroup, lower bounds, and make a link with the Riesz transform.
\bigbreak
\noindent
{\bf Home institution:}    \\[3mm]
Institut de Math\'ematiques de Bordeaux \\ 
Universit\'e de Bordeaux, UMR 5251,  \\ 
351, Cours de la Lib\'eration  \\
33405 Talence. France.\\
Thomas.Cometx@math.u-bordeaux.fr
\bigbreak
\noindent{\bf Acknowledgments}: This research is partly supported by the ANR project RAGE "Analyse Réelle et Géométrie" (ANR-18-CE40-0012).
\clearpage

\section{Introduction}
In this paper, we study conical vertical square functionals in the framework of Riemmannian manifolds.
Let $M$ be a complete non compact Riemannian manifold. The Riemannian metric on $M$ induces a distance $d$ and a measure $\mu$. We denote by $\nabla$ the Levi-Civita connection or the gradient on functions. Let $L = \Delta + V$ be a Schrödinger operator with $V$ a function in $L^1_{loc}$. Except when specifically precised, $V$ is non-negative. The conical vertical square function associated with $L$ is defined by 
$$\G_L(f) (x) = \left( \int_0^\infty  \int_{B(x,t^{1/2})} |\nabla e^{-tL} f(y)|^2 +V |e^{-tL} f(y)|^2 \frac{\mathrm{d}t \mathrm{d}y}{Vol(y,t^{1/2})}\right)^{1/2}$$
where $B(x,t^{1/2})$ is the ball of center $x$ and radius $t^{1/2}$ and $Vol(x,t^{1/2})$ its volume. We consider the question of boundedness of $\G_L$ on $L^p(M)$. We also compare $\G_L$ with the vertical Littlewood-Paley-Stein functional 
$$H_L(f) (x) = \left( \int_0^\infty |\nabla e^{-tL}f(x)|^2 + V|e^{-tL} f(x)|^2 \mathrm{d}t \right)^{1/2}. $$
Both of these functionals were introduced in the 
Euclidean setting and $ L = \Delta$ by Stein in \cite{Stein-Article} where he proved their boundedness on $L^p$ for all $p \in (1,+\infty)$.
Similar functionals associated with divergence form operators $L = div(A \nabla .)$ on $\mathbb{R}^n$ have been considered by Auscher, Hofmann and Martell in \cite{Auscher}. They showed that 
$$\left\| \left(  \int_0^\infty  \int_{B(x,t^{1/2})} |\nabla e^{-tL} f(y)|^2  \frac{\mathrm{d}t \mathrm{d}y}{Vol(y,t^{1/2})} \right)^{1/2}\right\|_p \leq C \left\| f\right\|_p $$ for $p \in (p^-, \infty)$ where $p^- \leq 2$ is the infimum of $p$ such that $\{ \sqrt{t} \nabla e^{-tL} ,t  \geq 0 \}$ satisfies $L^p-L^2$ off-diagonal estimates. In particular, if $A$ is real then $p^- = 1$. Chen, Martell and Prisuelos-Arribas studied the case of degenerate elliptic operators in \cite{Chen-Martell}.
The vertical Littlewood-Paley-Stein was studied by Stein for the Laplace-Beltrami operator in \cite{Stein-Article,Stein-Book} where he prove the boundedness of $H_\Delta$ on $L^p$ for $p \in (1,2]$ without any assumption on the manifold, and for $p \in (2,\infty)$ in the case of compact Lie groups. In \cite{Coulhon-Duong-Li}, Coulhon, Duong and Li proved the weak type $(1,1)$ for $H_\Delta$ if the manifold satisfies the volume doubling property and $\Delta$ satifies a Gaussian upper estimate for its heat kernel. In \cite{Ouhabaz-LPS}, Ouhabaz proved that $H_L$ is always bounded on $L^p$ for $p \in (1,2]$ and is unbounded for $p$ large enough. Cometx studied the case of Schrödinger operators with signed potential in \cite{Cometx}.

\bigbreak
Concerning $\G_L$ in the Riemannian manifold setting, we show that the situation for $p \in (1, 2]$ and $p \in [2,+\infty)$ are different. If $p \in [2,+\infty)$, it is proved in \cite{Auscher} that the conical square functional is bounded in the $L^p$ norm by the vertical one. 
We prove that the conical square functional is bounded on $L^p$ for all $p \in [2,+\infty)$ provided the manifold satisfies the volume doubling property.

In contrast, the vertical Littlewood-Paley-Stein functional $H_L$ may be unbounded on $L^p$ for $p$ large enough (see \cite{Cometx-Ouhabaz}, Section 7). This shows that $H_L$ and $\G_L$ have different behaviours on $L^p$.
If $p \in (1,2]$, then $H_L$ is always bounded on $L^p$ for any complete Riemannian manifold. 

Following the proofs in \cite{Auscher} and \cite{Chen-Martell}, we show in the Riemannian manifold setting that $\G_L$ is bounded on $L^p$ provided $ \{ \sqrt{t}\nabla e^{-tL} \}$ and $ \{ \sqrt{t} \sqrt{V} e^{-tL} \}$  satisfy $L^p-L^2$ off-diagonal estimates. 
In particular, if in addition the heat kernel of $e^{-t\Delta}$ satisfies a Gaussian upper bound, then $\G_L$ is bounded on $L^p$ for all $p \in (1,+\infty)$.

We also introduce generalized conical square functions, inspired by the generalized Littlewood-Paley-Stein functionals in \cite{Cometx-Ouhabaz}, namely
$$\G_L^F (f)(x) = \left( \int_0^\infty \int_{B(x,t^{1/2})} |\nabla F(tL) f(y)|^2 + V|F(tL) f(y)|^2 \frac{\mathrm{d}t \mathrm{d}y}{Vol(y,t^{1/2})}\right)^{1/2},$$
for $F$ a bounded holomorphic function in some sector $\Sigma(\mu) = \{ z \neq 0 , |arg(z)| < \mu \}$ for a fixed $\mu \in (0,\pi/2)$. We assume that the manifold satisfies the volume doubling property and $F$ has sufficient decay at zero and at infinity, that is 
$$|F(z)| \leq C \frac{|z|^{\tau}}{1+|z|^{\tau + \delta}},$$ for $\delta > 1/2$ and $\tau > \frac{N-2}{4}$. Then $\G_L^F$ is bounded on $L^p$ for all $p \in [2,+\infty)$.

In addition to Schrödinger operators we also consider conical square functionals associated with the Hodge-de Rham Laplacian on $1$-differential forms. That is 
$$\vec{\G} (\omega) (x) = \left( \int_0^\infty \int_{B(x,t^{1/2})} |d^* e^{-t\LF} \omega(y)|^2 \frac{\mathrm{d}y \mathrm{d}t}{Vol(y,t^{1/2})} \right)^{1/2},$$
where $d^*$ is the adjoint of the exterior derivative $d$. We show again that if the manifold is doubling then $\vec{\G}$ is bounded on $L^p$ for all $p \in [2, \infty)$. This boundedness is rather suprizing since the semigroup $e^{-t\LF}$ may not be uniformly bounded on $L^p$ for $p \neq 2$ (see \cite{Chen-Magniez-Ouhabaz}). In the case $p<2$, then $\vec{\G}$ is bounded on $L^p$ under the assumptions that $M$ satisfies the volume doubling property and $\{ \sqrt{t} d^* e^{-t\LF} , t\ge 0\}$ satisfies $L^p-L^2$ off-diagonal estimates. 

We also consider  conical vertical square functions for Schrödinger operators with a potential $V$ which have a non-trivial negative part $V^-$ and also such functionals associated with the Poisson semigroup. In addition we give lower bounds and an application to the Riesz transform.
\bigbreak

\textbf{Notations.} Throughout this chapter, we denote by $p'  = \frac{p}{p-1}$ the dual exponent of $p \in [1, \infty]$. We denote by $C,C',c$ all inessential positive constants. Given a ball $B = B(x,r) \subset M$ and $\lambda > 0$, $\lambda B$ is the ball $B(x, \lambda r)$. For a ball $B$ and $j \geq 1$, $C_j(B)$ (or $C_j$) is the annulus $2^{j+1} B \backslash 2^j B$ and $C_0(B)$ is $B$.

 We recall that $M$ satisfies the volume doubling property if for all $x$ in $M$ and $r > 0$ one has 
\[
 Vol(x,2r) \leq C Vol(x,r) \tag{D}
\] for some constant $C >0$ independent of $r$ and $x$.
 This property self-improves in  
\[\label{chap5-doubling}
 Vol(x,\lambda r) \leq C \lambda^N Vol(x,r) \tag{D'}
\]
 for some constants $C$ and $N$ independent of $x, r$ and $\lambda \geq 1$.

The Hardy-Littlewood maximal operator $\M$ is defined by
$$\M(f)(x) = \sup_{r>0} \frac{1}{\mu(B(x,r))} \int_{B(x,r)} |f(y)| \mathrm{d}y.$$ 
Given $\mu \in (0,\pi)$, $\Sigma(\mu)$ is the angular sector $\{ z \neq 0 , |arg(z)|< \mu \}$ and $H^\infty(\Sigma(\mu))$ is the set of bounded holomorphic functions on $\Sigma(\mu)$.

\section{Conical square functionals}\label{chap5-section-conical}
As mentionned in the introduction, the conical vertical functional associated with the Laplace-Beltrami operator $\Delta$ is defined by
\begin{equation*}
\mathcal{G}_\Delta(f)(x) := \left(\int_0^\infty \int_{B(x,t^{1/2})} |\nabla e^{-t\Delta} f|^2 \dtent \right)^{1/2}.
\end{equation*}
The so-called conical horizontal square functional is defined by
\begin{equation*}
\mathcal{S}_\Delta(f)(x) := \left(\int_0^\infty \int_{B(x,t^{1/2})} \left|\frac{\partial }{\partial t} e^{-t\Delta} f\right|^2 t \dtent \right)^{1/2}.
\end{equation*} 
The functional $\mathcal{S}_\Delta$ is linked to the Hardy spaces $H^p_\Delta$. The space $H^p_\Delta$ is the completion of the set $\{f \in H_\Delta^2, \| \SSS_\Delta f \|_p < +\infty \}$ with respect to the norm $\| \SSS_\Delta f \|_p$. The norm on $H_\Delta^p$ is $\| f\|_{H_\Delta^p} = \| \SSS_\Delta f \|_p$. Here $H^2_\Delta$ is the closure of $R(\Delta)$ with respect to the $L^2$ norm. The boundedness of $\SSS_\Delta$ on $L^p$  is equivalent to the inclusion $L^p \subset H_\Delta^p$.
  The Hardy space is important in the study of singular integral operators such as the Riesz transform. We refer to \cite{Auscher-Russ-McIntosh,Chen-these,Devyver-Russ,Hofmann-Mayboroda-2011,Hofmann-Mayboroda-2009} for more on this topic.
  
\bigbreak
Similarly, for a Schrödinger operator $L = \Delta + V$ with $0 \leq V  \in L^1_{loc}$ we define
\begin{align}
\G_L(f)(x) &:= \left(\int_0^\infty \int_{B(x,t^{1/2})} |\nabla e^{-tL} f|^2 + V|e^{-tL} f|^2 \dtent \right)^{1/2}, \\
\SSS_L(f)(x) &:=  \left(\int_0^\infty \int_{B(x,t^{1/2})} \left|\frac{\partial}{\partial t} e^{-tL} f\right|^2 t \dtent \right)^{1/2}.
\end{align} 
For the Hodge-de Rham Laplacan $\vec{\Delta} = d d^* + d^* d$ on $1$-differential forms we define
\begin{align}
\vec{\mathcal{G}}(\omega)(x) &:= \left(\int_0^\infty \int_{B(x,t^{1/2})} |d^* e^{-t\vec{\Delta}} \omega|^2 \dtent \right)^{1/2},\\
\vec{\mathcal{S}}(\omega)(x) &:= \left(\int_0^\infty \int_{B(x,t^{1/2})} \left| \frac{\partial}{\partial t} e^{-t\vec{\Delta}} \omega\right|^2 t \dtent \right)^{1/2}. \end{align}
Note that here we may also consider variants where one replaces $d^*$ by  the exterior derivative $d$ or by the Levi-Civita connection $\nabla$.

As in the case of the Laplace-Beltrami operator $\Delta$ on functions, one can define the Hardy spaces $H^p_L$ and $H_{\LF}^p$ throught $\SSS_L$ and $\vec{\SSS}$. See again \cite{Auscher-Russ-McIntosh,Chen-these,Devyver-Russ,Hofmann-Mayboroda-2011,Hofmann-Mayboroda-2009}.

We note that $\SSS_L$ is a particular case of square functions
$$\SSS_\phi(f)(x) := \left(\int_0^\infty \int_{B(x,t^{1/2}))} |\phi(tL) f|^2 \frac{\mathrm{d}y \mathrm{d}t}{tVol(y,t^{1/2})} \right)^{1/2},$$ where $\phi$ is a bounded holomorphic function on some angular sector . These ones are comparable with horizontal square functions associated to $L$ (see Proposition \ref{chap5-prop2-3}).

Following \cite{Auscher}, we define 
\begin{equation}\label{chap5-area-function}
A(F)(x) := \left( \int_0^\infty \int_{B(x,t)} |F(y,t)|^2 \frac{\mathrm{d}y \mathrm{d}t}{t Vol(y,t)} \right)^{1/2}
\end{equation}
and \begin{equation}\label{chap5-vertical-function} 
\tilde{V}(F)(x) := \left( \int_0^\infty  |F(y,t)|^2  \frac{\mathrm{d}t}{t} \right)^{1/2}.
\end{equation}
for any function $F$ which is locally square integrable on $M^+ := M \times \mathbb{R}_+$.
The functions $\tilde{V}(F)$ and $A(F)$ are measurable on $M$ and they are comparable in the following sense.

\begin{proposition}[\cite{Auscher}, Proposition 2.1]  \label{chap5-comparaison} Assume that $M$ satisfies the doubling volume property \eqref{chap5-doubling}. For every $F$ in $L^2_{loc}(M^+)$ we have
\begin{enumerate}\item For $p \in [2,+\infty)$,
$\|A (F)\|_p \leq C \|\tilde{V}(F)\|_p $.
\item For $p \in (0,2] $,
$\| \tilde{V} (F)\|_p \leq C \|A (F)\|_p$.
\end{enumerate}
\end{proposition}

\begin{remark}
In \cite{Auscher}, counter-examples for the reverse inequalities are given.
\end{remark}

Recall the vertical Littlewood-Paley-Stein functional is 
$$H_L (f)(x) = \left( \int_0^\infty |\nabla e^{-tL} f|^2 + V |e^{-tL} f|^2 \mathrm{d}t\right)^{1/2}.$$
As a corollary of Proposition \ref{chap5-comparaison} we have.
\begin{proposition} \label{chap5-prop2-3}
\begin{enumerate}
\item For $p \in [2,+\infty)$, $$ \|\G_L (f)\|_p \leq C \left\| H_L(f) \right\|_p.$$
\item Let $p \in  [2,+\infty)$ and $\phi$ be a bounded holomorphic function on the angular sector $\Sigma(\theta) := \{z \neq 0, |arg(z)| < \theta \}$ with $\theta \in (\arcsin \left| \frac{2}{p}-1\right|,\pi/2)$ such that $|\phi(z)| \leq C \frac{|z|^\alpha}{1 + |z|^{2\alpha}}$ for some $\alpha >0$ and all $z \in \Sigma(\theta)$. Then $\|\SSS_\phi f\|_p \leq C \|f \|_p$.
\end{enumerate}
\end{proposition}
\begin{proof}
The first item is an immediate consequence of Proposition \ref{chap5-comparaison} with $F(x,t) = |t \nabla e^{-t^2 \Delta} f|$. 
For the second one, using again Proposition \ref{chap5-comparaison} we obtain 
\begin{equation*}
\| \SSS_\phi (f) \|_p \leq C \left\| \left( \int_0^\infty |\phi(tL) f |^2 \frac{\mathrm{d}t}{t} \right)^{1/2} \right\|_p.
\end{equation*}
Since $L$ is the generator of a sub-Markovian, it has a bounded holomorphic functional calculus on $L^p$ for all $p \in (1, \infty)$. This was proved by many authors and the result had successive improvements during several decades. The most recent and general result in this direction states that $L$ has a bounded holomorphic functional calculus with angle $ \mu_p = \arcsin(|\frac{2}{p} - 1| ) + \epsilon$ for all $\epsilon > 0$. We refer to \cite{Carbonaro} for the precise statement. The existence of a bounded holomorphic functional calculus implies the so-called square functions estimates, that is for all $F \in H_0^\infty (\Sigma({\mu_p})) = \{F \in H^\infty (\Sigma({\mu_p})), |F(z)| \leq C \frac{|z|^\alpha}{1+|z|^{2\alpha}}$ for some $\alpha > 0$ and all $z$ in $\Sigma({\mu_p} ) \}$, one has for all $f$ in $L^p(M),$
$$\left\| \left( \int_0^\infty |F(tL)f(x)|^2 \frac{\mathrm{d}t}{t}\right)^{1/2} \right\|_p \leq C \| f\|_p. $$ See \cite{CDMY} for more on the link between square functions estimates and bounded holomorphic functional calculus. The square functions estimate with $F = \phi$ finishes the proof.
\end{proof}
\begin{remark}
The first item of the last proposition shows that if the Littlewood-Paley-Stein functional $H_L$ is bounded on $L^p$, then $\G_L$ is also bounded on $L^p$. Note that $H_L$ is bounded on $L^p$ for some $p \in [2,\infty)$ if and only if the sets $\{ \sqrt{t} \sqrt{V} e^{-tL} \}$ and $\{ \sqrt{t} \nabla e^{-tL} \}$ are $R$-bounded on $L^p$ (see \cite{Cometx-Ouhabaz}, Theorem 3.1).

\end{remark}
A natural choice for $\phi$ is $\phi_0(z) = z^{1/2} e^{-z}$ so that
\begin{equation}\label{chap5-S12}
\SSS_{\phi_0}(f)(x) := \left( \int_0^\infty \int_{B(x,t^{1/2})} | \Delta^{1/2} e^{-t\Delta} f|^2 \frac{\mathrm{d}y \mathrm{d}y}{Vol(y,t)}\right)^{1/2}.
\end{equation} 
We shall use this functional in Section \ref{chap5-section-riesz} in connection with the Riesz transform. We make the following observation.

\begin{proposition}
\begin{enumerate}
\item For $p \in [2, \infty)$, $\SSS_{\phi_0}$ is bounded on $L^p$,
\item For $p \in (1, 2]$, there exists $C>0$ such that for all $f \in L^{p}$, 
$$ \|f \|_{p} \leq C \| S_{\phi_0} (f)\|_{p}. $$ 
\end{enumerate}
\end{proposition}
\begin{proof}
The first item follows from Proposition \ref{chap5-prop2-3}.
For the second, fix $p \in (1,2]$, then $p'\in [2, \infty)$.
For all $f$ in $L^{p}$ and $g \in L^{p'}$ one has 
\begin{align*}
\left| \int_M f(x) g(x) \mathrm{d}x \right| &= \left|\int_M \int_0^\infty -\frac{\partial}{\partial t}( e^{-t\Delta} f e^{-t\Delta} g) \mathrm{d}t \mathrm{d}x  \right|\\
&= \left|\int_M \int_0^\infty  \left[ \Delta e^{-t\Delta}f e^{-t\Delta} g + e^{-t\Delta}f \Delta e^{-t\Delta} g \right] \mathrm{d}t \mathrm{d}x \right|\\
&=  2  \left|\int_M \int_0^\infty \Delta^{1/2} e^{-t\Delta} f . \Delta^{1/2} e^{-t\Delta} g \mathrm{d}t \mathrm{d}x\right| \\
&=  2 \left| \int_M \int_0^\infty \int_{y \in B(x,t^{1/2})} \Delta^{1/2} e^{-t\Delta} f . \Delta^{1/2} e^{-t\Delta} g \mathrm{d}t \mathrm{d}x \frac{\mathrm{d}y}{Vol(x,t^{1/2})}\right| \\
&= 2 \left| \int_M \left( \int_0^\infty \int_{x \in B(y,t^{1/2})}  \Delta^{1/2} e^{-t\Delta} f . \Delta^{1/2} e^{-t\Delta} g \mathrm{d}t \frac{\mathrm{d}x}{Vol(x,t^{1/2})} \right) \mathrm{d}y \right| \\
&\leq 2 \left| \int_M \SSS_{\phi_0}(f)(y) \SSS_{\phi_0}(g)(y) \mathrm{d}y \right| \\
& \leq 2 \| \SSS_{\phi_0}(g)\|_{p'} \|\SSS_{\phi_0}(f)\|_{p} \\
& \leq 2 \| g \|_{p'}  \|\SSS_{\phi_0}(f)\|_{p}.
\end{align*}
Here the two first inequalities respectively come from Cauchy-Schwarz with measure $\frac{\mathrm{d}t \mathrm{d}x}{Vol(x,t^{1/2})}$ and Hölder with exponents $p$ and $p'$. The last inequalities comes from the first item. We obtain the result by taking the supremum over $f$ in $L^p$.\end{proof}

\section{Tent spaces and off-diagonal $L^p-L^2$ estimates}\label{chap5-section-tentspaces}
In this short section, we recall the definition of tent spaces on manifolds some properties they satisfy.
For any $p \in [1,+\infty)$, the tent space $T_2^p$ is the space of square locally integrable functions on $M^+$  such that $$A(F) := \left( \int_0^\infty \int_{B(x,t)} |F(x,t)|^2 \frac{\mathrm{d}x \mathrm{d}t}{Vol(x,t)} \right)^{1/2} \in L^p(M).$$  Its norm is given by 
$$\| F\|_{T_2^p} = \| A(F) \|_p. $$

For $p = + \infty$, $T_2^\infty$ is the set of locally square integrable functions on $M^+$ such that
\begin{equation*}
\|F\|_{T_2^\infty} := \left( \sup_B \int_0^{r_B} \int_B |F(y,t)|^2 \frac{\mathrm{d}y \mathrm{d}t}{Vol(y,t)} \right)^{1/2} < + \infty.
\end{equation*} Here the supremum is taken on all balls $B$ in $M$ and $r_B$ is the radius of $B$. 

Tent spaces form a complex interpolation family and are dual of each other. Theses results remain true for tent spaces on mesured metric spaces with doubling volume property. In particular it is true for tent spaces of differential forms. We refer to  \cite{Chen-these} or \cite{Auscher-Russ-McIntosh} for proofs and more information. Precisely,

\begin{proposition}\label{chap5-interpolation}
Suppose $1 \leq p_0 < p < p_1 \leq \infty$, with $\frac{1}{p} = \frac{1-\theta}{p_0} + \frac{\theta}{p_1}$ for some $\theta \in (0,1)$. Therefore $[T_2^{p_0}, T_2^{p_1} ]_\theta = T_2^p.$
\end{proposition}

\begin{proposition}
Let $p$ be in $(1,+\infty)$ and $p'$ be its dual exponent. Then $T^{p'}_2$ is identified as the dual of $T_2^p$ with the pairing $<F,G> = \int_{M \times (0,+\infty)} F(x,t) G(x,t) \frac{\mathrm{d}x \mathrm{d}t}{t}.$
\end{proposition}

We shall use Proposition \ref{chap5-interpolation} to prove the boundedness of the conical square functions on $L^p$. Actually, the boundedness on $L^p$ of $\G_L$ canonically reformulates as the boundedness of $f \mapsto   t \nabla e^{-t^2 L} f $  and $ f \mapsto t \sqrt{V} e^{-t^2 L}f$ from $L^p$ to $T_2^p$. For $p \in [2,+\infty)$ the strategy is \begin{enumerate}
\item Prove that $\G_L$ is bounded on $L^2$,
\item Prove that $f \mapsto t \nabla e^{-t^2 L}f$  and $ f \mapsto t \sqrt{V} e^{-t^2 L}f$ are bounded from $L^\infty$ to $T_2^\infty$,
\item Deduce by interpolation that $\G_L$ is bounded on $L^p$ for all $p \in [2,+\infty)$.
\end{enumerate}
We use the same strategy for $\G_L^F$ and $\G_{\vec{\Delta}}$ in the forthcoming sections.
\bigbreak

In order to prove the boundedness of $f \mapsto t \nabla e^{-t^2 L}f$  and $f \mapsto t \sqrt{V} e^{-t^2 L}f$ from $L^\infty$ to $T_2^\infty$, we need Davies-Gaffney estimates for $\sqrt{t} \nabla e^{-tL}$ and $V^{1/2} \sqrt{t} e^{-tL}$. One says that a family $T_z$ of operators satisfies Davies-Gaffney estimates if for all $f$ in $L^2(M)$ and all closed disjoint sets $E$ and $F$ in $M$,
\begin{equation}\label{chap5-DG}
\| T_z (f \chi_E) \|_{L^2(F)} \leq C e^{-d^2(E,F)/|z|}\| f\|_{L^2(E)}.
\end{equation}

In \cite{Auscher-AMS} and \cite{Auscher}, the authors show that a good condition to prove the boundedness of conical square functions on $L^p$ for $p \in (1,2]$ is $L^p - L^2$ off-diagonal estimates for a well chosen family of operators. Let $ 1 \leq p \leq q < + \infty$. We say that a family $(T_t)_{t \geq 0}$ of operators satisfies $L^p - L^q$ off-diagonal estimates if for any ball $B$ with radius $r_B$ and for any $f$,

\begin{equation}\label{chap5-LpLq}
\left( \int_{C_j(B)}  |T_t f \chi_B |^q \mathrm{d}x \right)^{1/q} \leq   \frac{C} {\mu(B)^{\frac{1}{p}- \frac{1}{q}}} \sup \left( \frac{2^{j}r_B}{\sqrt{t}}, \frac{\sqrt{t}}{ 2^{j} r_B} \right)^\beta e^{-  c 4^j r_B^2/t} \left( \int_{B}  | f |^p \mathrm{d}x \right)^{1/p}.
\end{equation}


We mostly use the case $q=2$, that is 
\begin{equation}\label{chap5-LpLq2}
\left( \int_{C_j(B)}  |T_t f \chi_B |^2 \mathrm{d}x \right)^{1/2} \leq  \frac{C }{\mu(B)^{\frac{1}{p}- \frac{1}{2}}} \sup  \left( \frac{2^j r_B}{\sqrt{t}}, \frac{\sqrt{t}}{2^j r_B} \right)^\beta  e^{-  c 4^j r_B^2/t} \left( \int_{B}  | f |^p \mathrm{d}x \right)^{1/p},
\end{equation}
for all $j \geq 1$ and some $\beta, C >0$ independent of $B$, $j$ and $f$.
Here $C_j (B) = 2^{j+1} B \backslash 2^j B$. One can also consider analytic families of operators and then one can write the previous inequalities for $z$ in some sector $\Sigma(\mu) = \{z \neq 0, |arg(z)| < \mu \}$ for a given $\mu \in (0,\pi/2)$.

In several cases, the uniform boundedness of the semigroup on $L^p$ for implies that $\sqrt{t} \nabla e^{-tL}$ satisfies \eqref{chap5-LpLq2}.
This is the case if the manifold has the volume doubling property \eqref{chap5-doubling} and its heat kernel associated with $\Delta$ satisfies the Gaussian upper estimate \eqref{chap5-Gaussian}. Recall that the heat kernel $p_t$ associated with $\Delta$ satisfies the Gaussian upper estimate \eqref{chap5-Gaussian} if there exist constants $C,c>0$ such that the heat kernel $p_t$ satisfies for all $x,y \in M$
\[\label{chap5-Gaussian}
p_t(x,y) \leq C \frac{e^{-c d^2(x,y)/t}}{Vol(y,t^{1/2})}. \tag{G}
\]
For $L^p-L^q$ off-diagonale estimates for Schrödinger operators on manifolds with subcritical negative part of the potential, see \cite{Assaad-Ouhabaz}. In the case of the Hodge-de Rham operatorn, see Section \ref{chap5-section6}, or \cite{Magniez}.

\section{Study of $\G_L$ }
In this section, $L = \Delta + V$ is a Schrödinger operator with $ 0 \leq V \in L^1_{loc}$. We make some remarks about the case of a signed potentiel at the end of the section. Recall that $\G_L$ is defined by
$$\G_L(f)(x) = \left( \int_0^\infty \int_{B(x,t^{1/2})} |\nabla e^{-t L} f(y)|^2  + V |e^{-tL} f(y)|^2\frac{\mathrm{d}t \mathrm{d}y}{ Vol(y,t^{1/2})}\right)^{1/2}.$$
In this section, we prove the boundedness of $\G_L$ on $L^p(M)$ under some assumptions depending on $p \in (1, 2]$ or $p \in [2,+\infty)$. In the framework of second order divergence form operators $L = div (A \nabla .)$ on $\mathbb{R}^d$, it has been proven in \cite{Auscher} that $\G_L$ is bounded on $L^p$ for all $p \in (1,+\infty)$ and of weak type $(1,1)$ if $A$ is real.

This functional is easier to study for $p \in [2,\infty)$ and its boundedness comes from an argument from \cite{Fefferman-Stein}. The only assumption we need on the manifold here is the volume doubling property \eqref{chap5-doubling}. We start by the boundedness on $L^2$.
\begin{proposition}\label{chap5-theo-G-pgrand}
$\G_L$ is bounded on $L^2$.
\end{proposition}
\begin{proof}
We compute 
\begin{align*}
\|\G_L(f)\|_2^2 &= \int_M \int_0^\infty \int_{B(x,t^{1/2})} |\nabla e^{-tL} f(y)|^2 + V|e^{-tL}f(y)|^2 \frac{\mathrm{d}y \mathrm{d}t \mathrm{d}x}{Vol(y,t^{1/2})} \\
&= \int_M \int_0^\infty \int_{B(y,t^{1/2})} |\nabla e^{-tL} f(y)|^2 + V|e^{-tL}f(y)|^2 \frac{\mathrm{d}x \mathrm{d}t \mathrm{d}y}{Vol(y,t^{1/2})} \\
&= \int_M \left(\int_0^\infty |\nabla e^{-tL} f(y)|^2 + V|e^{-tL}f(y)|^2 \int_{B(y,t^{1/2})} 1 \mathrm{d}x \mathrm{d}t \right)\frac{\mathrm{d}y}{Vol(y,t^{1/2})} \\
&= \int_M \left(\int_0^\infty |\nabla e^{-tL} f(y)|^2 + V|e^{-tL}f(y)|^2 \mathrm{d}t\right)\mathrm{d}y \\
&=  \int_M \int_0^\infty (\Delta + V) e^{-tL} f(y)\cdot e^{-tL} f(y)\mathrm{d}t \,\mathrm{d}y \\
&=\frac{1}{2} \|f \|_2^2.
\end{align*} \end{proof}
For $p \in [2, \infty)$, we have the following theorem.

\begin{theorem}\label{chap5-theogrand}
If $M$ satisfies the doubling volume property \eqref{chap5-doubling}, then $\G_L$ is bounded on $L^p$ for all $p \in [2, \infty)$.
\end{theorem}
\begin{proof}
Let $\Gamma$ be either $\nabla$ or the multiplication by $\sqrt{V}$. We show that $f \mapsto t\Gamma e^{-t^2 L} f$ is bounded from $L^\infty$ to $T_2^\infty$. By interpolation it is bounded from $L^p$ to $T_2^p$ for all $p \in [2,\infty]$, what reformulates as the boundedness of $\G_L$ on $L^p$.

Recall that the norm on $T_2^\infty$ is given by $$\| F \|_{T_2^\infty} = \left( \sup_B \frac{1}{\mu(B)}\int_B \int_0^{r_B} |F(x,t)|^2 \frac{\mathrm{d}x\mathrm{d}t}{t} \right)^{1/2}$$
where the supremum is taken over all balls $B$ in $M$ and $r_B$ is the radius of $B$. Fix a ball $B$ and decompose $f = f \chi_{4B} + f \chi_{(4B)^c}$.
For the local part $ f \chi_{4B} $ we have
\begin{align*}
 \frac{1}{\mu(B)} \int_B \int_0^{r_B} |t \Gamma e^{-t^2 L} f \chi_{4B} |^2 \frac{\mathrm{d}x \mathrm{d}t}{t} &\leq \frac{C}{\mu(B)} \left\| \left(\int_0^\infty |\Gamma  e^{-tL} f \chi_{4B}|^2  \mathrm{d}t \right)^{1/2} \right\|_2^2 \\&\leq \frac{C}{\mu(B)} \| f \chi_{4B}\|_2^2 \\&\leq C  \|f\|_\infty^2.
\end{align*}
We now deal with the non-local part. We decompose $f \chi_{(4B)^c} = \sum_{j\geq 2} f \chi_{C_j}$, where $C_j(B) = 2^{j+1} B \backslash 2^j B$. Davies-Gaffney estimates \eqref{chap5-DG} for $ \sqrt{t} \nabla e^{-tL}$ give
\begin{align*}
&\left( \frac{1}{\mu(B)} \int_0^{r_B} \int_B  |t\Gamma e^{-t^2 L} \sum_{j \geq 2} f \chi_{C_j} |^2 \frac{\mathrm{d}x \mathrm{d}t}{t} \right)^{1/2} \\&\hspace{4cm} \leq C \sum_{j\geq 2}  \left( \int_0^{r_B} \int_{C_j} \frac{e^{\frac{-4^j r_B^2}{t^2}}\mu(C_j)}{\mu(B)\mu(C_j)} |f|^2 \frac{\mathrm{d}x \mathrm{d}t}{t} \right)^{1/2} \\
&\hspace{4cm}\leq C   \sum_{j \geq 2} \left( \frac{2^{jN}}{\mu(C_j)} \int_0^{r_B} e^{\frac{-4^j r_B^2}{t^2}}  \frac{\mathrm{d}t}{t}  \int_{C_j} |f|^2  \mathrm{d}x \right)^{1/2}\\
&\hspace{4cm} \leq C  \| f\|_\infty.
\end{align*}
We obtain that $f \mapsto t \Gamma e^{-t^2} f$ is bounded from $L^\infty$ to $T_2^\infty$.
It is then bounded from $L^p$ to $T_2^p$ for all $p \in [2,\infty] $ by interpolation. This gives that $\G_L$ is bounded on $L^p$.
We see this by writing
\begin{align*}
\G_L(f)(x) &= \left( \int_0^\infty \int_{B(x,t^{1/2})} |\Gamma e^{-tL} f|^2+ V |e^{-tL} f|^2 \frac{\mathrm{d}y \mathrm{d}t}{Vol(y,t^{1/2})} \right)^{1/2} \\
&=\frac{1}{2} \left( \int_0^\infty \int_{B(x,s)} | s \Gamma e^{-s^2L} f|^2 +  V |s e^{-s^2 L} f|^2  \frac{\mathrm{d}y \mathrm{d}s}{s Vol(y,s)} \right)^{1/2}\\
&= \frac{1}{2} A(F)(x)
\end{align*} where $F(x,s) = \left(|s \nabla e^{-s^2 L} f|^2 + |s V e^{-s^2 L} f|^2 \right)^{1/2}$.
Then  \begin{equation*}
\| \G_L(f) \|_p = \frac{1}{2}\|F\|_{T_2^p} \\ \leq C \|f\|_p. \end{equation*} \end{proof}

\begin{remark}We give two examples which show that the Littlewood-Paley-Stein functional and the conical square functional have different behaviors for $p \in [2, \infty)$.
\begin{enumerate}
\item In $\mathbb{R}^d$, under reasonable assumptions (see \cite{Ouhabaz-LPS}), if $V$ is not identically equal to zero, then $H_L$ is unbounded on $L^p$ for $p > d$, whereas $\G_L$ is bounded.
\item Let $M$ be the connected sum of two copies of $\mathbb{R}^d$ glued among the unit circle. The Littlewood-Paley-Stein functional $H_\Delta$ is unbounded on $L^p$ for $p \in (d,+\infty)$ whereas $\G_\Delta$ is bounded (see \cite{Carron-Coulhon-Hassell}).
\end{enumerate}
\end{remark}
The case $p  \in (1, 2]$ is more difficult. We have to assume off-diagonal $L^p-L^2$ estimates for the gradient of semigroup, namely
\begin{multline}\label{chap5-LpL2gradient}
\| \sqrt{t}\nabla e^{-t L} f \|_{L^2(C_j)}  + \| \sqrt{t}\sqrt{V} e^{-t L} f \|_{L^2(C_j)} \\ \leq \frac{C}{\mu(B)^{1/p - 1/2}} \sup( \frac{2^j r}{\sqrt{t}},\frac{\sqrt{t}}{2^j r})^\beta e^{-4^j {r_B}^2/t} \|f\|_{L^p(B)}.
\end{multline}Note that these estimates are always true in the case of $\mathbb{R}^n$ if $V \geq 0$. 
For a signed potential $V = V^+ - V^-$, the discussion is postponed to the end of the section.
\begin{theorem}\label{chap5-theo-G-ppetit} Assume that $M$ satisfies the doubling property \eqref{chap5-doubling} and $\{ \sqrt{t} \nabla e^{-tL} \}$ and $\{ \sqrt{t} \sqrt{V} e^{-tL} \}$ satisfy $L^p-L^2$ off diagonal estimates \eqref{chap5-LpL2gradient} for some $p \in [1,2)$. Then $\G_L$ is of weak type $(p,p)$ and bounded on $L^q$ for all $ p < q \leq 2$.
\end{theorem}

\begin{remark}
The proof is the same as in \cite{Chen-Martell} where the authors deal with divergence form operators on $\mathbb{R}^n$. We reproduce the details for the sake of completeness. We write down the proof for the gradient part $$\G_L^{(\nabla)}(f)(x) = \left( \int_0^\infty \int_{B(x,t^{1/2})} |\nabla e^{-tL} f|^2 \frac{\mathrm{d}y \mathrm{d}t}{Vol(y,t^{1/2})} \right)^{1/2}.$$
The proof is the same for the part with $\sqrt{V}$.
\end{remark}

\begin{proof}
Fix $p \in [1,2)$. $\G_L$ is bounded on $L^2(M)$, then by the Marcinkiewickz interpolation theorem it is enough to prove that $\G_L$ is of weak type $(p,p)$. Fix $\lambda > 0$ and $f \in L^p$, we use the $L^p$ Calderon-Zygmund decomposition (see \cite{Chen-Martell} or \cite{Stein-Book-2}) of $f$ by writing $f = g + \sum_i b_i$ where \begin{enumerate}
\item $(B_i)_{i\geq 1}$ is sequence of balls of radius $r_i>0$ in $M$ such that the sequence $(4 B_i)_{i\geq 1}$  has finite overlap number, that is $\sup_{x \in M} \sum_{i \geq i} \chi_{4 B_i}(x) < \infty$,
\item $|g| \leq C \lambda$ almost everywhere,
\item The support of $b_i$ is included in $B_i$ and $\int_{B_i} |b_i|^p \mathrm{d}x \leq C \lambda \mu(B_i)$,
\item $\sum_i \mu(B_i) \leq \frac{C}{\lambda^p} \int_M |f(x)|^p \mathrm{d}x$.

\end{enumerate}

For simplicity, we write down the proof in the case $p = 1$. It is the same for any $p \in (1,2)$. Set $A_{r_i} := I - (I-e^{-r_i^2 L})^K$ for $K$ a positive integer to be chosen.
One has \begin{align*}\mu(\{x : \G_L^{(\nabla)}(f)(x) < \lambda \}) &\leq \mu(\{x : \G_L^{(\nabla)}(g)(x) < \lambda/3 \})  \\ &+  \mu( \{ x: \G_L^{(\nabla)}(\sum A_{r_i} b_i)(x) < \lambda/3 \} ) \\
&+ \mu(\{ x :  \G_L^{(\nabla)}(\sum (I-e^{-r_i^2 L})^K b_i)(x) < \lambda/3\}) \\ &=: I + II + III. 
\end{align*}
Using the boundedness of $\G_L^{(\nabla)}$ on $L^2$ and the properties of the Calderon-Zygmund decomposition, it is a classical fact that $ I \leq \frac{C}{\lambda} \| f \|_1$. 
It remains to estimate $II$ and $III$. We first estimate $II$. Take $0 \leq  \psi \in L^2(M)$ with norm $\|\psi\|_2 = 1$. One has 
$$\int_M \left| \sum A_{r_i} b_i (x) \right| \psi(x) \mathrm{d}x \leq \sum_{i \geq 1} \sum_{j \geq 0} \left( \int_{C_j(B_i)} |A_{r_i} b_i|^2 \mathrm{d}x \right)^{1/2} \left( \int_{2^{j+1} B_i} \psi^2 \mathrm{d}x \right)^{1/2}.$$
We note that $A_{r_i}$ satisfies $L^p - L^2$ estimates \eqref{chap5-LpLq2}. The notation we use is 
\begin{equation}
\| A_{r_i} f \|_{L^2(C_j)} \leq \frac{C }{\mu(B)^{1/2}} \sup( 2^j,2^{-j})^\gamma e^{-c 4^j } \|f\|_{L^1(B)} \end{equation}
for some $\gamma >0$. It leads to
\begin{align*}
\int_M \left| \sum_{i \geq 1} A_{r_i} b_i\right| \psi \mathrm{d}x & \leq  \sum_{ i \geq 1} \sum_{j \geq 0}  \frac{C \mu(2^{j+1} B)^{1/2}}{\mu(B)^{1/2}}e^{-c4^j} \left[ \sup ( 2^j,2^{-j})\right]^\gamma \\
&\times\left( \int_{B_i} |b_i| \mathrm{d}x \right) \inf_{B_i} \M(\psi^2)^{1/2}(x)\\
&\leq \lambda \int_{\cup_i B_i} \M(\psi^2)^{1/2}(x) \mathrm{d}x \\
&\leq \lambda \mu(\bigcup_i B_i)^{1/2} \| \psi\|_2 \\
&\leq C \lambda^{1/2} \| f\|_1^{1/2}. 
\end{align*}
Since $\sum_i A_{r_i} b_i$ is in $L^2$, the boundedness of $\G_L^{(\nabla)}$ gives $II  \leq C \frac{1}{\lambda} \| f\|_1$.
The two last inequalities come from Jensen and the boundedness of $\M$. Since $\sum_i A_{r_i} b_i$ is in $L^2$, the boundedness of $\G_L^{(\nabla)}$ on this space gives $II \leq \frac{C}{\lambda} \| f\|_1$. Finally, we estimate $III$. Markov inequality gives
\begin{align*}
III & \leq  \mu\left( \bigcup_i 5 B_i\right) + \mu\left(\{x \in M \backslash \bigcup_i 5 B_i, \G_L^{(\nabla)}(\sum_i (I-e^{-r_i^2 L})^K b_i)(x) \geq \lambda/4 \}\right) \\
&\leq C \left[ \frac{1}{\lambda} \|f\|_1 + \frac{1}{\lambda^2}  \int_{M \backslash \bigcup_i 5 B_i}  \G_L^{(\nabla)}(\sum_i (I-e^{-r_i^2 L})^K b_i)^2(x) \mathrm{d}x \right].
\end{align*}
Set $h_i := (I-e^{-r_i^2 L})^K b_i$. One has
\begin{align*}
&\int_{M \backslash \bigcup_i 5 B_i}  \G_L^{(\nabla)}(\sum_i  h_i)^2(x) \mathrm{d}x  \\
&\leq C \int_0^\infty \int_M \left|\sum_i \chi_{4 B_i}(y) t \nabla e^{-t^2 L} h_i \right|^2 \mu(B(y,t) \backslash \bigcup 5 B_i) \frac{\mathrm{d}x \mathrm{d}t}{t Vol(y,t)} \\
&+ C  \int_0^\infty \int_M \left|\sum_i \chi_{M \backslash 4 B_i}(y) t \nabla e^{-t^2 L} h_i \right|^2 \mu(B(y,t) \backslash \bigcup 5 B_i) \frac{\mathrm{d}x \mathrm{d}t}{t Vol(y,t)} \\
&=: C \left[ K_{loc} + K_{glob} \right].
\end{align*}
We start by estimating $K_{loc}$. Given $y \in 4 B_i$, if there exists $x \in B(y,t) \backslash  \bigcup_i 5 B_i$, then  $t >r_i$. Therefore,
\begin{align*}
K_{loc} &\leq C \sum_{i=1}^\infty \int_{r_i}^\infty \int_{4 B_i} \left| t\nabla e^{-t^2 L} h_i(y) \right|^2 \mu(B(y,t) \backslash \bigcup_i 5 B_i )\frac{\mathrm{d}y \mathrm{d}t}{Vol(y,t)} \\
&\leq C  \sum_{i=1}^\infty \int_{r_i}^\infty \int_{4 B_i} \left| t\nabla e^{-t^2 L} h_i(y) \right|^2 \mathrm{d}y \mathrm{d}t.
\end{align*}
The off-diagonal estimates \eqref{chap5-LpL2gradient} give
\begin{align*}
 \left( \int_{4 B_i} \left| t\nabla e^{-t^2 L} (h_i(y) \chi_{4Bi}) \right|^2 \mathrm{d}y \right)^{1/2} &\leq \frac{C}{\mu(4B_i)^{1/2}}  \left( \frac{r_i}{t}\right)^{\beta} \int_{4 B_i} \left| h_i(y) \right|\mathrm{d}y \\
 &\leq \frac{C}{\mu(4B_i)^{1/2}}  \left( \frac{r_i}{t}\right)^{\beta} \int_{4 B_i} \left| b_i(y) \right|\mathrm{d}y  \\
 &\leq \frac{\mu(B_i)}{\mu(4B_i)^{1/2}}  \left( \frac{r_i}{t}\right)^{\beta} \lambda \\
 &\leq \mu(B_i)^{1/2} \left( \frac{r_i}{t}\right)^{1/2} \lambda.
\end{align*}
By the same arguments and expending $(I  -e^{r_i^2 L})^M$ we obtain 
\begin{align*}
 &\left(\int_{4 B_i} \left| t\nabla e^{-t^2 L} h_i(y) \chi_{(4Bi)^c}) \right|^2 \mathrm{d}y \right)^{1/2} \\&\leq \left(  \int_{4 B_i} \left|\sum_{j\geq 2 } t\nabla e^{-t^2 L} h_i(y) \chi_{C_j} \right|^2 \mathrm{d}y \right)^{1/2} \\
 &\leq  \sum_{j\geq 2 }  \left( \int_{2^{j+1} B_i} \left| t\nabla e^{-t^2L} h_i(y) \chi_{C_j} \right|^2 \mathrm{d}y  \right)^{1/2}\\
 &\leq C  \sum_{j\geq 2 } \frac{2^{j\beta}}{\mu(2^{j+1} B_i)^{1/2}} \left(\frac{r_i}{t}\right)^{\beta} \sum_{k=1}^M \left( \int_{C_j(B_i)} \left| e^{-kr_i^2 L} b_i \right| \mathrm{d}y \right)  \\
 &\leq C \sum_{j\geq 2 } \frac{2^{j(\beta+\gamma)}}{\mu(2^{j+1}B_i)^{1/2}} \left(\frac{r_i}{t}\right)^{\beta}   e^{-c4^j}\left( \int_{ B_i} \left|b_i \right| \mathrm{d}y \right).
  \end{align*}
  The properties of the Calderon-Zygmund decomposition and the volume doubling property \eqref{chap5-doubling} give
 \begin{align*}
 \sum_{j\geq 2 } \frac{2^{j(\beta+\gamma)}}{\mu(2^{j+1}B_i)^{1/2}} \left(\frac{r_i}{t}\right)^{\beta}   e^{-c4^j}\left( \int_{ B_i} \left|b_i \right| \mathrm{d}y \right)
 &\leq C \lambda \sum_{j\geq 2 } \frac{2^{j(\beta+\gamma)}\mu(B_i)}{\mu(2^{j+1}B_i)^{1/2}} \left(\frac{r_i}{t}\right)^{\beta}   e^{-c4^j} \lambda \\
 &\leq C \lambda \mu(B_i)^{1/2} \left(\frac{r_i}{t}\right)^{\beta} .
\end{align*}
By the properties of the Calderon-Zygmund decomposition again we have
\begin{align*}
K_{loc} &\leq C \lambda^2 \sum_i \mu (B_i) \int_{r_i}^\infty \left(\frac{r_i}{t}\right)^{2\beta} \frac{\mathrm{d}t}{t} \\
&\leq C \lambda^2 \sum_i \mu(B_i)\\
&\leq C \lambda \|f \|_1.
\end{align*}
Finally, we deal with $K_{glob}$. Take $\Phi \geq 0$ in $L^2(M^+,\frac{\mathrm{d}y \mathrm{d}t}{t})$ with norm $\|\Phi \|_2 = 1$. Set $$\tilde{\Phi}(y) := \int_0^\infty \Phi(y,t)^2 \frac{\mathrm{d}t}{t}.$$ We have
\begin{align*}
\int_0^\infty &\int_M \left|\sum_{i \geq 1}\chi_{(4B_i)^c}(y) t \nabla  e^{-t^2 L} h_i(y) \right| \Phi(y,t) \frac{\mathrm{d}y \mathrm{d}t}{t} \\ 
&= \int_0^\infty \int_M \left|\sum_{i \geq 1} \sum_{j \geq 2} \chi_{C_j(B_i)}(y) t \nabla  e^{-t^2 L} h_i(y) \right| \Phi(y,t) \frac{\mathrm{d}y \mathrm{d}t}{t}\\
&\leq C \sum_{i \geq 1} \sum_{j \geq 2}\left( \int_0^\infty \int_{C_j(B_i)} |t \nabla e^{-t^2 L} h_i(y) |^2 \right)^{1/2} \left( \int_0^\infty \int_{C_j(B_i)} \Phi(y,t)^2 \frac{\mathrm{d}y \mathrm{d}t}{t} \right)^{1/2}
\\
&\leq C \sum_{i \geq 1} \sum_{j \geq 2} I_{i,j} \mu(C_j(B_i))^{1/2} \inf_{x \in B_i} (\M(\tilde{\Phi})(x))^{1/2}
\end{align*}
where $$I_{i,j} = \left( \int_0^\infty  \int_{C_j(B_i)} |t \nabla e^{-t^2 L}h_i(y)|^2 \mathrm{d}y \frac{\mathrm{d}t}{t} \right)^{1/2} \leq C \mu(B_i)^{1/2} 2^{-j(2K)}$$ by Lemma \ref{chap5-lemme-preuve-chen} below. Therefore,
\begin{align*}
\int_0^\infty &\int_M \left|\sum_{i \geq 1}\chi_{(4B_i)^c}(y) t \nabla  e^{-t^2 L} h_i(y) \right| \Phi(y,t) \frac{\mathrm{d}y \mathrm{d}t}{t} \\  
&\leq C \lambda \sum_{i \geq 1}  \sum_{j \geq 2} \mu(B_i)^{1/2} \mu(C_j(B_i))^{1/2} 2^{-2jK} \inf_{x \in B_i} (\M(\tilde{\Phi})(x))^{1/2}  \\
& \leq C \lambda \sum_{i \geq 1}  \sum_{j \geq 2} \mu(B_i)2^{-j(2K - N/2)} \inf_{x \in B_i} (\M(\tilde{\Phi})(x))^{1/2}.
\end{align*}
Choosing $K > N/4$ gives
\begin{align*}
\lambda \sum_{i \geq 1}  \sum_{j \geq 2} \mu(B_i)2^{-j(2K - N/2)} \inf_{x \in B_i} (\M(\tilde{\Phi})(x))^{1/2} &\leq C \lambda \sum_{i \geq 1} \mu(B_i)  \inf_{x \in B_i} (\M(\tilde{\Phi})(x))^{1/2}\\
&\leq C \lambda \int_{\bigcup B_i} (\M ( \tilde{\Phi}))^{1/2} \mathrm{d}x \\
&\leq C \lambda \mu (\bigcup B_i)^{1/2} \\
&\leq \lambda^{1/2} \|f \|_1^{1/2}.
\end{align*}
Here the last inequality comes from the properties of the Calderon-Zygmund decomposition. Hence, $III \leq \lambda^{-1} \|f \|_1$ and we obtain the result.
\end{proof}
In the proof, we use the following lemma which follows from functional calculus on $L^2(M)$ (see \cite{Chen-Martell}).
\begin{lemma}\label{chap5-lemme-preuve-chen}
For any $i \geq 1$ and $j \geq 2$, 
$$I_{i,j}= \left( \int_0^\infty  \int_{C_j(B_i)} |t \nabla e^{-t^2 L}h_i(y)|^2 \mathrm{d}y \frac{\mathrm{d}t}{t} \right)^{1/2} \leq C \mu(B_i)^{1/2} 2^{-j(2K)}.$$
\end{lemma}

The classical setting of doubling manifolds with an heat kernel satisfying a Gaussian upper estimates is covered by the theorem.

\begin{corollary}
Assume that $M$ satisfies the doubling property \eqref{chap5-doubling} and that the heat kernel associated with $\Delta$ satisfies the Gaussian upper estimate \eqref{chap5-Gaussian}. Then $\G_L$ is bounded on $L^p$ for all $p \in (1,+\infty)$.
\end{corollary}
\begin{proof}
Assume that $M$ satisfies the doubling volume property \eqref{chap5-doubling} and that the heat kernel associated with $\Delta$ satisfies the Gaussian upper estimate \eqref{chap5-Gaussian}. Then $\{ \sqrt{t }\nabla e^{-tL} \}$ and $ \{ \sqrt{t} \sqrt{V} e^{-tL} \}$ both satisfy $L^p-L^2$  estimates for all $p \in [1,2]$. Hence, $\G_L$ is bounded on $L^p$ for all $p \in (1,2]$ by Theorem \ref{chap5-theo-G-ppetit}. The case $p \in (2,+\infty)$ comes from Theorem \ref{chap5-theogrand}.
\end{proof}

In the case of Schrödinger operator with signed potential $ L = \Delta + V^+ - V^-$, we can state similar results. The conical vertical square functional for $L$ is defined by
$$\G_L(f)(x) = \left( \int_0^\infty \int_{B(x,t^{1/2})} |\nabla e^{-tL} f(y)|^2 + |V| |e^{-tL}f(y)|^2 \frac{\mathrm{d}y \mathrm{d}t}{Vol(y,t^{1/2})} \right)^{1/2}.$$

\begin{theorem}\label{chap5-theo4-8}Assume that $M$ satisfies the doubling property \eqref{chap5-doubling}. Suppose that $V^-$ is subcritical with respect to $\Delta + V^+$, that is there exists $\alpha \in (0,1)$ such that for all smooth and compactly supported function $f$,
\begin{equation}\label{chap5-souscritique}\int_M V^- f^2 \mathrm{d}x \leq \alpha \int_M V^+ f^2 +|\nabla f|^2 \mathrm{d}x.  \end{equation}
Then, \begin{enumerate}
\item $\G_L$ is bounded on $L^p$ for all $p \in [2, \infty)$.
\item Assume in addition that the kernel associated with $e^{-t\Delta}$ satisfies the Gaussian upper estimate \eqref{chap5-Gaussian}. If $N \leq 2$, then $\G_L$ is bounded for all $p \in (1,+\infty)$. If $N > 2$, set $p'_0 = \frac{2}{1-\sqrt{1-\alpha}}\frac{N}{N-2}$. Then $\G_L$ is bounded for all $p \in ({p_0},+\infty)$.\end{enumerate}
\end{theorem}

\begin{proof} Let $p$ be in $(1, 2]$ if $N \leq 2$ or in $({p_0},2]$ otherwise. In \cite{Assaad-Ouhabaz} the authors prove that, under the assumptions of the theorem, both $\{\sqrt{t} \nabla e^{-tL} \}$ and $\{\sqrt{t} |V|^{1/2} e^{-tL} \}$ satisfy Davies-Gaffney estimates \eqref{chap5-DG} and off-diagonal estimates \eqref{chap5-LpLq2}. The same proof as in the case of a non-negative potential applies and gives the boundedness of $\G_L$.
\end{proof}

\section{Generalized conical square functions associated with Schrödinger operators}

In this section, we introduce generalized conical square functions for Schrödinger operators $L= \Delta + V$ with $0 \leq V \in L^1_{loc}$. Let   $F$ be an holomorphic function in $H^\infty(\Sigma(\mu))$, with $ \Sigma(\mu)= \{ z \neq 0, |arg(z)| < \mu \}$ for some $\mu \in (\mu_p,\pi/2)$. We have already mentioned and used that $L$ has a bounded holomorphic functional calculus with angle $\mu \in (\mu_p = \arcsin |\frac{2}{p} - 1|, \pi/2 )$ on $L^p(M)$ for $p \in (1,+\infty)$. In particular, $F(L)$ is a bounded operator on $L^p(M) $ for $F \in H^\infty(\Sigma(\mu))$. We define $\G^F_{L}(f)$ by

$$\G^F_{L} (f)(x) = \left( \int_0^\infty \int_{B(y,t^{1/2})} |\nabla F (tL) f(y)|^2 + V |F(tL) f(y)|^2 \frac{\mathrm{d}t \mathrm{d}y}{Vol(y,t^{1/2})}\right)^{1/2}.$$
We start by the case $p=2$.
\begin{proposition}Assume there exist $C, \epsilon >0$ and $\delta > 1/2$ such that  $|F(z)|\leq \frac{C}{|z|^{\delta}}$ as $|z| \rightarrow + \infty$ and $|F'(z)|\leq \frac{C}{|z|^{1-\epsilon}}$ as $z \rightarrow 0$. Then
$\G_L^F$ is bounded in $L^2(M)$.
\end{proposition}
\begin{proof}The boundedness of $$f \mapsto \left( \int_0^\infty |\nabla F(tL) f|^2 + V |F(tL) f|^2 \mathrm{d}t \right)^{1/2} $$ on $L^2(M) $ from \cite{Cometx-Ouhabaz} (Theorem 4.1) gives
\begin{align*}
\| \G_L^F (f) \|_2^2 &=\left \|  \left( \int_0^\infty \int_{B(y,t^{1/2})} |\nabla F (tL) f(y) |^2  + V |F(tL)f (y)|^2\frac{\mathrm{d}t \mathrm{d}y}{Vol(y,t^{1/2})}\right)^{1/2} \right\|_2^2 \\
&=   \left \|  \left( \int_0^\infty  |\nabla F (tL) f|^2 + V |F(tL)|^2 \mathrm{d}t \right)^{1/2} \right\|_2^2 \\
&\leq C  \|f\|^2_2.
\end{align*}
 \end{proof}
Recall that a family $\{T_i, i \in I \}$ of operators is $R$-bounded on $L^p$ if there exists $C>0$ such that for all $n \in \mathbb{N}$ and all $i_1,...,i_n \in I$ and for all $f_1,...,f_n$ in $L^p$, $$\left\| \left( \sum_{i=1}^n |T_i f_i|^2 \right)^{1/2}\right\|_p \leq C \left\| \left( \sum_{i=1}^n |f_i|^2 \right)^{1/2} \right\|_p.$$ It is known from \cite{Cometx-Ouhabaz} that the $R$-boundedness is linked with the boundedness on $L^p$ of the Littlewood-Paley-Stein functionals. We have

\begin{proposition}\label{chap5-theo5-1}
Given $p \in [2,+\infty)$ and $F \in H^\infty(\Sigma(\mu))$ with $\mu \in (\mu_p, \pi/2)$. Assume that there exist $C,\epsilon >0 $ and $ \delta> 1/2$ such that $|F(z)| \leq \frac{C}{|z|^{\delta}}$ as $|z| \rightarrow \infty$ and $|F'(z)|\leq \frac{C}{|z|^{1-\epsilon}}$ as $z \rightarrow 0$. If the families $\{ \sqrt{t} \nabla e^{-tL} \}$ and $\{ \sqrt{t} \sqrt{V} e^{-tL} \}$ are $R$-bounded on $L^p(M)$, then $\G_L^F$ is bounded on $L^p$. 
\end{proposition}
\begin{proof} By Proposition \ref{chap5-comparaison}, one has 
\begin{align*}
\left\| \G_L^F (f) \right\|_p &\leq C \left\|  \left( \int_0^\infty |\Gamma F (t\Delta) f|^2 + V | F (t\Delta) f|^2  \mathrm{d}t \right)^{1/2} \right\|_p \\
&\leq C \left\| f\right\|_p.
\end{align*}
The last inequality comes from the $R$-boundedness of  $\{ \sqrt{t} \Gamma e^{-t\Delta} \}$ on $L^p(M)$ for either $\Gamma = \nabla$ or $\Gamma = \sqrt{V}$ (see \cite{Cometx-Ouhabaz}, Theorem 4.1).
\end{proof}

\begin{remark}
Let $\Gamma$ be either $\nabla$ or the multiplication by $\sqrt{V}$.
\begin{enumerate}
\item It follows from \cite{Cometx-Ouhabaz} (Proposition 2.1) that the boundedness of the Riesz transform $\Gamma L^{-1/2}$ on $L^p$ implies the $R$-boundedness of  $\{ \sqrt{t} \Gamma e^{-t\Delta} \}$.

\item One can generalize Proposition \ref{chap5-theo5-1} as in \cite{Cometx-Ouhabaz}. Consider $h_1, ... , h_n$ bounded holomorphic functions on $\Sigma(\mu) = \{ z \neq 0 , |arg(z)| < \mu \}$. Under the assumptions of Theorem \ref{chap5-theo5-1} there exists $C>0$ such that for all $f_1,...,f_n \in L^p(M)$, 
\begin{multline*}\left\| \left( \int_0^\infty \int_{B(.,t^{1/2})} \sum_{i=1}^n |\nabla h_i(L) F(tL) f_i(y)|^2 + V |h_i(L) F(tL) f_i(y)|^2 \frac{\mathrm{d}t \mathrm{d}y}{Vol(y,t^{1/2})}   \right)^{1/2} \right\|_p\\
\leq C \left\| \left( \sum_{i=1}^n  |f_i|^2 \right)^{1/2} \right\|_p .\end{multline*}
\end{enumerate}
\end{remark}

We state another positive result concerning the boundedness of $\G_L^F$, assuming the function $F$ has sufficient decay at zero and at infinity. We start by giving Davies-Gaffney estimates for $F(t 
L)$. This lemma is inspired by Lemma 2.28 in \cite{Hofmann-Mayboroda-McIntosh} where a similar result is proven for $F(tL)$ instead of $\sqrt{t} \Gamma F(tL)$.
\begin{lemma}\label{chap5-lemme-ft}
Let $\mu > 0$. Let $F$ be an holomorphic function on a the sector $\Sigma(\mu)$ such that there exist $\tau, \sigma > 0$ such that for all $z \in \Sigma(\mu)$, $|F(z)| \leq C \frac{|z|^\tau}{1+|z|^{\tau + \sigma}}$ Then for all $f \in L^2(M)$ and all disjoint closed subsets $E$ and $G$ of $M$,
\begin{equation}\label{chap5-offdiag-m}
\| \sqrt{t} \Gamma F(tL) f \chi_E \|_{L^2(G) }\leq C \left(\frac{t}{d(E,G)^2} \right)^{\tau+1/2}  \|f\|_{L^2(E)}.
\end{equation}
Here $\Gamma$ is either $\nabla$ or the multiplication by $\sqrt{V}$.
\end{lemma}

\begin{proof}
The functionnal calculus for $L$ on $L^2$ gives the representation formula
\begin{equation}\label{chap5-representation-formula}
 \Gamma F(tL) f = \int_{\Gamma_0^+} \Gamma e^{-zL} f \eta_+(z) dz + \int_{\Gamma_0^-} \Gamma e^{-zL} f  \eta_-(z) dz,\end{equation}where
$$\eta^{\pm}(z) = \frac{1}{2i \pi}\int_{\gamma^\pm}   e^{z \zeta} F(t \zeta)   d \zeta.$$
Here $\Gamma_0^{\pm} =\mathbb{R}_+ e^{ \pm i (\pi / 2 - \theta)}$ for some $\theta \in (0,\pi/2)$ and $\gamma^{\pm} = \mathbb{R}_+ e^{ \pm i \nu}$ for some $\nu < \theta$.
Under our assumption on $F$, we obtain 
\begin{align*}
|\eta^{\pm}| (z) &\leq  \int_{\gamma^\pm} |e^{\zeta z}| |F(t\zeta)| d \zeta \\
&\leq C \int_{\gamma^\pm} |e^{\zeta z}| \frac{ |t\zeta|^\tau}{1+ |t \zeta|^{\tau+\sigma}} d \zeta \\
&\leq C \left[ \int_{\zeta \in \gamma^\pm, |\zeta| \leq 1/t} |e^{\zeta z}| \frac{ |t\zeta|^\tau}{1+ |t \zeta|^{\tau+\sigma}} d \zeta  +  \int_{\zeta \in \gamma^\pm, |\zeta| > 1/t}  |e^{\zeta z}| \frac{ |t\zeta|^\tau}{1+ |t \zeta|^{\tau+\sigma}} d \zeta \right] \\
&:=C \left[ J_1 + J_2 \right]. 
\end{align*}
We bound 
\begin{align*}
J_1 &\leq C \int_{\zeta \in \gamma^\pm, |\zeta| \leq 1/t} e^{- \delta |z||\zeta|} \frac{ |t\zeta|^\tau}{1+ |t \zeta|^{\tau+\sigma}} d \zeta\\
&\leq C \frac{t^\tau}{|z|^{\tau+1}} \int_0^\infty e^{-\delta \rho} d \rho  \\
&\leq C \frac{t^\tau}{|z|^{\tau+1}}.
\end{align*}
Here $\delta \in (0,1)$ depends on $\theta$ and $\mu$. Besides,
\begin{align*}
J_2 &\leq C \int_{\zeta \in \gamma^\pm, |\zeta| > 1/t}  |z \zeta|^{-\tau-1} |t \zeta|^{-\sigma} d \zeta  \\
&\leq C \left( \frac{t}{|z|}\right)^{\tau+1} t^{-\tau-\sigma-1}
\int_{\zeta \in \gamma^\pm, |\zeta| > 1/t}  |\zeta|^{-\tau-\sigma-1} d \zeta \\
&\leq C \frac{t^\tau}{|z|^{\tau+1}}.
\end{align*}
Hence,
\begin{equation}\label{chap5-bilan-eta}
|\eta_\pm|(z) \leq C \frac{t^\tau}{|z|^{\tau+1}}.
\end{equation}
Then \eqref{chap5-representation-formula} and \eqref{chap5-bilan-eta} together give that for all $f$ in $L^2$ and all disjoints closed sets $E$ and $G$ in $M$,
\begin{equation*}
\| \Gamma F(tL) f \|_{L^2(G)} \leq C \left[ \int_{\Gamma_0^+} \| \Gamma e^{-zL}  f \|_{L^2(G)} \frac{t^\tau}{|z|^{\tau+1}} dz + \int_{\Gamma_0^-} \| \Gamma e^{-zL}  f \|_{L^2(G)} \frac{t^\tau}{|z|^{\tau+1}}dz  \right].
\end{equation*}

We bound the first term. The second is bounded by the same method. Davies-Gaffney estimates \eqref{chap5-DG} for $\{ \sqrt{z} \Gamma e^{-zL} \}$ give
\begin{align*}
\int_{\Gamma_0^+} \| \Gamma e^{-zL} f\|_{L^2(G)}   t^\tau |z|^{-\tau-1} dz &\leq C \left(\int_{\Gamma_0^+}   t^\tau |z|^{-\tau-3/2} e^{-c d(E,G)^2/|z|} dz \right) \|f \|_{L^2(E)} \\&\leq C t^\tau (d(E,G)^2)^{-\tau - 1/2} \left(\int_0^\infty   s^{-\tau-3/2} e^{-c/s} \mathrm{d}s \right) \|f \|_{L^2(E)} \\
&\leq \frac{C}{\sqrt{t}} \left( \frac{t}{d(E,G)^2} \right)^{\tau+1/2} \|f \|_{L^2(E)}.
\end{align*}
\end{proof}

As a consequence of these Davies-Gaffney estimates, we obtain the boundedness of generalized conical square functionals.
\begin{theorem}Assume that $M$ satisfies the doubling property \eqref{chap5-doubling}. Let $F$ be an holomorphic function on a sector $\Sigma(\mu) = \{z \neq 0, |arg(z)| < \mu\}$ such that for all $z$ in $\Sigma(\mu)$,
$|F(z)| \leq C \frac{|z|^\tau}{1 + |z|^{\tau+\delta}}$ for some $\tau >  (N-2)/4$ and $\delta > 1/2$, where $N$ is as in \eqref{chap5-doubling}.
Then $\G_L^F$ is bounded on $L^p$ for all $p \in [2,+\infty)$.
\end{theorem}
\begin{proof}
The boundedness of $\G_L^F$ on $L^2$ follows from Theorem \ref{chap5-comparaison} and \cite{Cometx-Ouhabaz}, Theorem 4.1.
 Let $\Gamma$ be either $\nabla$ or the multiplication by $\sqrt{V}$. We use the same proof as for Theorem \ref{chap5-theogrand} to prove that $f \mapsto t \Gamma F(t^2L) f$ is bounded from $L^\infty$ to $T_2^\infty$.
Recall that the norm on $T_2^\infty$ is given by $$\| F \|_{T_2^\infty} = \left( \sup_B \frac{1}{\mu(B)}\int_B \int_0^{r_B} |F(x,t)|^2 \frac{\mathrm{d}x\mathrm{d}t}{t}\right)^{1/2}$$
where the supremum is taken over all balls and $r_B$ is the radius of $B$. Fix a ball $B$ and decompose $f = f \chi_{4B} + f \chi_{(4B)^c}$.
We start by dealing with $ f \chi_{4B} $. One has \begin{align*}
\frac{1}{\mu(B)} \int_B \int_0^{r_B} |t \Gamma F(t^2 L) f \chi_{4B} |^2 \frac{\mathrm{d}x \mathrm{d}t}{t} & \leq \frac{1}{\mu(B)}\int_M \int_0^{\infty} |t\nabla F(t^2 L) f \chi_{4B} |^2 \frac{\mathrm{d}x \mathrm{d}t}{t} \\
&\leq \frac{1}{\mu(B)} \left\| \left(\int_0^\infty | \Gamma  F(t^2 L) f \chi_{4B}|^2 t \mathrm{d}t \right)^{1/2} \right\|_2^2.
\end{align*}
The boundedness of $f \mapsto \left(\int_0^\infty |\Gamma  F(s L) f \chi_{4B}|^2 \mathrm{d}s \right)^{1/2}$ on $L^2$ and the doubling property \eqref{chap5-doubling} give 
\begin{align*}
\frac{1}{\mu(B)} \int_B \int_0^{r_B} |t\Gamma F(t^2 L) f \chi_{4B} |^2 \frac{\mathrm{d}x \mathrm{d}t}{t} &\leq  \frac{1}{2\mu(B)} \left\| \left(\int_0^\infty |\Gamma  F(s L) f \chi_{4B}|^2 \mathrm{d}s \right)^{1/2} \right\|_2^2  \\
&\leq \frac{C}{\mu(B)} \| f \chi_{4B}\|_2^2 \\&\leq C  \|f\|_\infty^2.
\end{align*}
We now deal with the non-local part $f \chi_{(4B)^c}$. We decompose $f \chi_{(4B)^c} = \sum_{j\geq 2} f \chi_{C_j}$, where $C_j = 2^{j+1} B \backslash 2^j B$. Lemma \ref{chap5-offdiag-m} and the doubling volume property \eqref{chap5-doubling} yield
\begin{align*}
&\left( \frac{1}{\mu(B)}\int_B |t\Gamma F(t^2 L) \sum_{j\geq 2} f \chi_{C_j} |^2 \mathrm{d}x \right)^{1/2} 
\\& \hspace{4cm}\leq \sum_{j\geq 2} \left( \frac{1}{\mu(B)}\int_B |t \Gamma F(t^2L)  f \chi_{C_j} |^2 \mathrm{d}x \right)^{1/2}
\\& \hspace{4cm}\leq C \sum_{j\geq 2} \frac{t^{2\tau+1}\mu(C_j)^{1/2}}{\mu(B)^{1/2} \mu(C_j)^{1/2} r^{2\tau+1}4^{j(\tau+1/2)}} \left( \int_{C_j} f^2 \mathrm{d}x \right)^{1/2}\\
& \hspace{4cm}\leq C \sum_{j\geq 2} \frac{2^{jN/2} t^{2\tau+1}}{\mu(C_j)^{1/2}r^{2\tau} 4^{j\tau}} \left( \int_{C_j} f^2 \mathrm{d}x \right)^{1/2} \\
& \hspace{4cm}\leq C \sum_{j\geq 2} \frac{2^{jN/2} t^{2\tau+1}}{r^{2\tau+1} 4^{j(\tau+1/2)}}  \| f\|_\infty \\
& \hspace{4cm}\leq C \frac{t^{2\tau+1}}{r^{2\tau+1}} \|f \|_\infty.
\end{align*}
The convergence of the sum comes from the choice $\tau > (N-2)/4$. Therefore,
\begin{align*}
\frac{1}{\mu(B)} \int_0^{r_B} \int_B |t\Gamma F(t^2 L) f_{\chi_{C_j}}|^2 \frac{\mathrm{d}x \mathrm{d}t}{t} &\leq C\|f \|_\infty^2 \int_0^{r_B} \frac{t^{4\tau+1}}{r^{4\tau+2}} \mathrm{d}t \\
&\leq C \|f\|_\infty^2.
\end{align*}
Hence $\| t \Gamma F(t^2 L) f \|_{T_2^\infty} \leq  C \| f \|_p.$ By interpolation, we obtain that $f \mapsto t \Gamma F(t^2 L) f$ is bounded from $L^p$ to $T_2^p$ for all $p>2$. This gives the boundedness of $\G_L^F$ on $L^p$. Indeed,
\begin{align*}
\G_L^F(f)(x) &= \left( \int_0^\infty \int_{B(x,t^{1/2})} |\nabla F(t L) f|^2 + V |F(tL) f|^2 \frac{\mathrm{d}y \mathrm{d}t}{Vol(y,t^{1/2})} \right)^{1/2} \\
&= \frac{1}{2} \left( \int_0^\infty \int_{B(x,s)} | s \nabla F(s^2 L) f|^2 + V |s F(s^2 L) f|^2\frac{\mathrm{d}y \mathrm{d}s}{s Vol(y,s)} \right)^{1/2}\\
&= \frac{1}{2} A(\Psi)(x)
\end{align*} where $\Psi(x,s) = \left( |s \Gamma F(s^2 L) f|^2 +  V |s F(s^2L) f|^2 \right)^{1/2}$.
Then $\| \G_L^F(f) \|_p = \frac{1}{2} \|\Psi\|_{T_2^p} \leq C \|f\|_p $. 
\end{proof}

\begin{remark}\begin{enumerate}
\item This result still holds replacing $F(tL)$ by $h(L) F(tL)$ where $h$ is holomorphic and bounded. Actually, for all $f$ in $L^p$ we have
 $$ \left\| \left( \int_0^\infty \int_{B(y,t^{1/2})} |\nabla h(L) F(tL) f |^2 + |\sqrt{V} h(L) F(tL) |^2 \frac{\mathrm{d}y \mathrm{d}t}{Vol(y,t^{1/2})} \right)^{1/2}\right\|_p \leq \|f \|_p. $$
\item If $V$ is a signed potential with subcritical negative part, we obtain the boundedness of $\G_L^F$ on $L^p$ for all $p \in (2, \infty)$ whereas the semigroup does not acts boundedly on $L^p$ for $p$ large enough. It follows from the fact that the family $\{\sqrt{z} \Gamma e^{-zL}\}$ satisfies Davies-Gaffney estimates \eqref{chap5-DG} under the assumption of subcriticality \eqref{chap5-souscritique} (see \cite{Assaad-Ouhabaz}).
\end{enumerate}
\end{remark}

\section{Study of $\vec{\G}$}\label{chap5-section6}
The vertical conical square function assiocitated with $\vec{\Delta}$ is defined by
$$\vec{\G}(\omega)(x) = \left( \int_0^\infty \int_{B(x,t^{1/2})} |d^* e^{-t\LF} \omega|_x^2 \frac{\mathrm{d}y \mathrm{d}t}{Vol(y,t^{1/2})}\right)^{1/2}. $$
In this section, we apply the same techniques as for $\G_L$ to obtain the boundedness of $\vec{\G}$. The following lemma, from \cite{Auscher-Russ-McIntosh}, says that $d^* e^{-t\vec{\Delta}}$ satisfies Davies-Gaffney estimates.
\begin{lemma}[\cite{Auscher-Russ-McIntosh}, Lemma 3.8]
The family $\sqrt{t} d^* e^{-t \vec{\Delta}}$ satisfies Davies-Gaffney estimates, that is for all closed sets $E$ and $F$ and for any differential form $\omega$ in $L^2$, \begin{equation}\label{chap5-DGformes}
\| d^* e^{-t\vec{\Delta}}  \omega \chi_E \|_{L^2(F)} \leq \frac{C}{\sqrt{t}}e^{-c d^2(E,F)/t }\| \omega\|_{L^2(E)}.
\end{equation}
\end{lemma}
This lemma implies the boundedness of $\vec{\G}$ on $L^p$ for all $p \in [2,+\infty)$.
\begin{theorem}Assume that $M$ satisfies the doubling volume property \eqref{chap5-doubling}, then $\vec{\G}$ is bounded on $L^p(\Lambda^1 T^* M)$ for all $p \in [2,+\infty)$.
\end{theorem}
\begin{proof}
The proof is the same as for $\G_L$. We reproduce it for the sake of completeness. As for $\G_L$, Proposition \ref{chap5-comparaison} gives that $\vec{\G}$ is bounded on $L^2$ because the Littlewood-Paley-Stein  functional
$$ \omega \mapsto \left(\int_0^\infty |d^* e^{-t\LF} \omega|^2_x \mathrm{d}t \right)^{1/2} $$is bounded on $L^2$. We  show that $\omega \mapsto t d^* e^{-t^2 \vec{\Delta}} \omega $ is bounded from $L^\infty$ to $T_2^\infty$. By interpolation it is bounded from $L^p$ to $T_2^p$ for all $p>2$, what reformulates as the boundedness of $\vec{\G}$ on $L^p$. For interpolation of tent spaces, we refer to Lemma \ref{chap5-interpolation} which remains true in the case of tent spaces of differential forms.

Recall that the norm on $T_2^\infty$ si given by $$\| F \|_{T_2^\infty} =  \left( \sup_B \frac{1}{\mu(B)}\int_B \int_0^{r_B} |F(x,t)|^2 \frac{\mathrm{d}x\mathrm{d}t}{t} \right)^{1/2} $$
where the supremum is taken over all balls $B$ with radius $r_B$. Fix a ball $B$ and decompose $\omega = \omega \chi_{4B} + \omega \chi_{(4B)^c}$. One has
\begin{align*}
\frac{1}{\mu(B)} \int_B \int_0^{r_B} |t d^* e^{-t^2 \vec{\Delta}} \omega \chi_{4B} |^2 \frac{\mathrm{d}x \mathrm{d}t}{t} &\leq \frac{1}{\mu(B)} \left\| \left( \int_0^\infty |d^* e^{-t\LF} \omega \chi_{4B} |^2 \mathrm{d}t \right)^{1/2} \right\|_2^2 \\&\leq \frac{1}{2 \mu(B)} \| \omega \chi_{4B}\|_2^2 \\ &\leq C \|\omega\|_\infty^2.
\end{align*}
We decompose $\omega \chi_{(4B)^c} = \sum_{j\geq 2} \omega \chi_{C_j}$, where $C_j = 2^{j+1} B \backslash 2^j B$. Minkowski inequality and Davies-Gaffney estimates \eqref{chap5-DGformes} give
\begin{align*}
&\left(\frac{1}{\mu(B)}\int_0^{r_B}\int_B |t d^* e^{-t^2 \vec{\Delta}} \sum_{j \geq 2} \omega \chi_{C_j} |^2 \frac{\mathrm{d}x \mathrm{d}t}{t} \right)^{1/2} \\&\leq C \sum_{j \geq 2} \left( \int_0^{r_B} \frac{e^{\frac{- c 4^j r_B^2}{t^2}}\mu(C_j)}{\mu(B)\mu(C_j)} \int_{C_j} |\omega|^2 \frac{\mathrm{d}x \mathrm{d}t}{t}  \right)^{1/2}\\
&\leq C \sum_{j \geq 2}\left( \int_0^{r_B} \frac{2^{jN} e^{\frac{- c 4^j r_B^2}{t^2}}}{\mu(C_j)}  \int_{C_j} |\omega|^2 \frac{\mathrm{d}x \mathrm{d}t}{t} \right)^{1/2}\\
&\leq C \sum_{j \geq 2} \left( \int_0^{r_B} 2^{jN} e^{\frac{-c 4^j r_B^2}{t^2}} \frac{\mathrm{d}t}{t}\right)^{1/2}  \| \omega\|_\infty \\
&\leq C  \| \omega\|_\infty.
\end{align*}
Then $\| t d^* e^{-t^2 \LF} \omega \|_{T_2^\infty} \leq C \| \omega \|_\infty$.
By interpolation we obtain that $\omega \mapsto t d^* e^{-t^2 \LF} \omega$ is bounded from $L^p$ to $T_2^p$ for all $p \in [2,\infty]$, what reads as the boundedness of $\vec{\G}$ on $L^p$.
Indeed,
\begin{align*}
\vec{\G}(\omega)(x) &= \left( \int_0^\infty \int_{B(x,t^{1/2})} |d^* e^{-t\vec{\Delta}} \omega|^2\frac{\mathrm{d}y \mathrm{d}t}{Vol(y,t^{1/2})} \right)^{1/2} \\
&= \frac{1}{2}  \left( \int_0^\infty \int_{B(x,s)} | s d^* e^{-s^2\vec{\Delta}} \omega|^2\frac{\mathrm{d}y \mathrm{d}s}{s Vol(y,s)} \right)^{1/2}\\
&= \frac{1}{2} A(\Psi)(x)
\end{align*} where $\Psi(x,s) = s \nabla e^{-s^2 \vec{\Delta}} \omega$.
Therefore we have $\| \vec{\G}(\omega) \|_p = \frac{1}{2} \|\Psi\|_{T_2^p} \leq C \|\omega\|_p $. \end{proof}

These case $p \in (1,2)$ is more complicated. Following the proof of Theorem \ref{chap5-theo-G-ppetit}, we have the following result.

\begin{theorem}\label{chap5-SC-formes} Assume that $M$ satisfies the doubling property \eqref{chap5-doubling} and that the set $ \{ \sqrt{t} d^* e^{-t\vec{\Delta}} \}$ satisfies $L^p-L^2$ off-diagonal estimates \eqref{chap5-LpLq2} for some $p < 2$. Then $\vec{\G}$ if of weak type $(p,p)$ and is bounded on $L^q$ for all $ p < q \leq 2$.
\end{theorem}
As for Schrödinger operator, we can state positive results assuming smallness of the negative part of the Ricci curvature.

\begin{theorem}\label{chap5-ricci-subcritical-case}
Assume that $M$ satisfies the doubling property \eqref{chap5-doubling} and that the kernel associated with $\Delta$  satisfies a Gaussian upper estimate. Assume in addition that ${R}^-$ is subcritical with respect to $\nabla^* \nabla  + R^+$, that is there exists $\alpha \in (0,1)$ such that for all $\omega \in C^\infty_0(\Lambda^1 T^* M)$,
$$\int_M <R^- \omega, \omega> \mathrm{d}x \leq \alpha \int_M <R^+ \omega, \omega> +|\nabla \omega|^2 \mathrm{d}x. $$ 
If $N \leq 2$, then $\vec{\G}$ is bounded for all $p \in (1,+\infty)$. If $N > 2$, let $p'_0 = \frac{2}{1-\sqrt{1-\alpha}}\frac{N}{N-2}$. Then $\vec{\G}$ is bounded for all $p \in ({p_0},+\infty)$.
\end{theorem}
\begin{proof} The Gaussian upper estimate \eqref{chap5-Gaussian}, the doubling volume property \eqref{chap5-doubling} together with the subcriticality condition imply that $\sqrt{t} d^* e^{-t\vec{\Delta}}$ satisfies the $L^p-L^2$ estimates \eqref{chap5-LpLq2} (see \cite{Chen-Magniez-Ouhabaz}, Theorem 4.6). We apply Theorem \ref{chap5-SC-formes} to conclude.
\end{proof}

\section{Conical square functionals associated with the Poisson semigroup}
In \cite{Auscher}, the authors also introduce the conical square functionals associated for the Poisson semigroup associated with divergence form operators on $\mathbb{R}^d$. For a Schrödinger operator $L = \Delta + V$ with a potential $0 \leq V \in L^1_{loc}$, we define similar functionals by
$$P_L(f)(x) = \left( \int_0^\infty \int_{B(x,t)} |\nabla_{t,y} e^{-tL^{1/2}} f|^2  + V |e^{-tL^{1/2}} f|^2 \frac{ t \mathrm{d}t \mathrm{d}y}{Vol(y,t)} \right)^{1/2}.$$
We denote by $P_{L,t}$ the time derivative part of $P$ and $P_{L,x}$ the gradient part. If $V=0$, we denote them respectively by $P, P_t$ and $P_x$. 
\begin{align*}
P_{L,x}(f)(x)&= \left( \int_0^\infty \int_{B(x,t)} |\nabla_y e^{-tL^{1/2}} f|^2 + V |e^{-tL} f|^2 \frac{ t \mathrm{d}t \mathrm{d}y}{Vol(y,t)} \right)^{1/2},\\
P_{L,t}(f)(x) &= \left( \int_0^\infty \int_{B(x,t)} \left|\frac{\partial}{\partial t} e^{-tL^{1/2}} f\right|^2 \frac{ t \mathrm{d}t \mathrm{d}y}{Vol(y,t)} \right)^{1/2}.
\end{align*}
We ask whether $P_L$ is bounded or not on $L^p$. We start by the case $p=2$.

\begin{proposition}
$P_L$ is bounded on $L^2(M)$.
\end{proposition}
\begin{proof}
One has 
\begin{align*}
\| P_L(f) \|_2^2 &= \int_M \int_0^\infty \int_{y\in B(x,t)} |\nabla_y e^{-tL^{1/2}} f|^2 + \left|\frac{\partial}{\partial t} e^{-tL^{1/2}} f \right|^2 + V |e^{-tL^{1/2}}f |^2 \frac{t \mathrm{d}t \mathrm{d}y \mathrm{d}x}{Vol(y,t)} \\
&=\int_M \int_0^\infty |\nabla_y e^{-tL^{1/2}} f|^2 + \left|\frac{\partial}{\partial t} e^{-tL^{1/2}} f\right|^2 + V |e^{-tL^{1/2}} f|^2 t \mathrm{d}y \mathrm{d}t \\
&= 2 \int_0^\infty \frac{\partial}{\partial t} \| e^{-tL^{1/2}} f\|^2_2 \mathrm{d}t \\
&= 2 \| f\|_2^2.
\end{align*}
\end{proof}
\begin{remark}
The pointwise equality $P_L(f) = (P_{L,x}^2(f) + P_{L,t}(f))^{1/2}$ gives that $P_{L,t}$ and $P_{L,x}$ are bounded on $L^2$.
\end{remark}

In order to study the case $p \in [2,+\infty)$, we compare $P_L$ and $\G_L$. We start by the following technical lemma concerning the volume of the balls.

\begin{lemma}\label{chap5-lemme-derivation}
Assume that $M$ satisfies the volume doubling property \eqref{chap5-doubling}, then 
$|\nabla_{t,y} Vol(y,t)| \leq C t^{-1} Vol(y,t).$
\end{lemma}
\begin{proof}
We start by the time derivative part. For all $h>0$, one has by the doubling property \eqref{chap5-doubling}
\begin{align*}
Vol(y,t+h) - Vol(y,t) &\leq C \left( (\frac{t+h}{t})^N  - 1 \right)Vol(y,t) \\
&= C \left ( (1+\frac{h}{t})^N - 1 \right) Vol(y,t) \\
&\leq C h t^{-1} Vol(y,t).
\end{align*}
For the gradient part we have 
\begin{align*}
\frac{Vol(z,t)-Vol(y,t)}{d(z,y)} &\leq C \frac{ Vol(y,t+d(x,y)) - Vol(y,t) }{ d(z,y)}\\
&\leq C \left( (\frac{d(z,y)+t}{t})^N-1\right) \frac{ Vol(y,t)}{d(z,y)} \\
&\leq C \left( (\frac{d(z,y)+t}{t})^N-1\right) \frac{ Vol(y,t)}{d(z,y)} \\
&= C \left( (1+\frac{d(z,y)}{t})^N-1\right) \frac{ Vol(y,t)}{d(z,y)}
 \\&\leq C t^{-1} Vol(y,t).\end{align*}
\end{proof}
The following lemma from \cite{Auscher} will also be useful to study to compare $P_L$ and $\G_L$.
\begin{lemma}\label{chap5-lemme 6}
For any $f \in L^2$ and $x \in M$ one has 
\begin{multline}
P_L(f)(x) \leq C \left[ \left( \int_0^\infty \int_{B(x,2t)} \left| \left( e^{-t^2 L} - e^{-tL^{1/2}}\right) f \right|^2 \frac{\mathrm{d}y \mathrm{d}t}{t Vol(y,t)}\right)^{1/2} \right.\\
+ \left. \left(  \int_0^\infty \int_{B(x,2t)} |\nabla_{t,y} e^{-t^2 L} f|^2 + V |e^{-t^2 L} f|^2 \frac{ t \mathrm{d}y \mathrm{d}t}{Vol(y,t)} \right)^{1/2}  \right].
\end{multline}
\end{lemma}
\begin{proof}
We note that 
$$P_L(f)(x) \leq  \left( \int_0^\infty \int_M \left[ |\nabla_{t,y} e^{-tL^{1/2}} f|^2  + V |e^{-tL^{1/2}}f|^2 \right] \phi^2  \left(\frac{d(x,y)}{t}\right) \frac{ t \mathrm{d}t \mathrm{d}y}{Vol(y,t)} \right)^{1/2}$$
where $\phi$ is a non-negative smooth function on $\mathbb{R}_+$ such that $\phi(s) = 1$ if $s \leq 1$ and $\phi(s) = 0$ if $s>2$.
Set $u := e^{-tL^{1/2}} f$ and $ v := e^{-t^2 L} f$. One has 
\begin{align*}
P_L(f)(x)^2 &\leq \int_M \int_0^\infty \left[\nabla_{t,y} u . \nabla_{t,y}  (u-v) + V u (u-v)\right] \phi^2 \left(\frac{d(x,y)}{t}\right) \frac{ t \mathrm{d}t \mathrm{d}y}{Vol(y,t)} \\&+ \int_M \int_0^\infty \left[ \nabla_{t,y}  u . \nabla_{t,y}  v + V uv \right]  \phi^2 \left(\frac{d(x,y)}{t} \right) \frac{ t \mathrm{d}t \mathrm{d}y}{Vol(y,t)} \\
&=: I_1 +I_2.  
\end{align*}
By Cauchy-Schwarz and Young inequalities we obtain for all $\epsilon>0,$
\begin{align*}
I_2  &\leq   \epsilon\int_0^\infty  \int_{B(x,2t)} |\nabla_{t,y} u|^2   \phi^2 \left(\frac{d(x,y)}{t}\right) \frac{t \mathrm{d}t \mathrm{d}y}{Vol(y,t)} \\ &+ \epsilon^{-1} \int_0^\infty   \int_{B(x,2t)} |\nabla_{t,y} v|^2   \phi^2 \left(\frac{d(x,y)}{t}\right) \frac{t \mathrm{d}t \mathrm{d}y}{Vol(y,t)} \\&+  \epsilon\int_0^\infty  \int_{B(x,2t)} V u^2   \phi^2 \left(\frac{d(x,y)}{t}\right) \frac{t \mathrm{d}t \mathrm{d}y}{Vol(y,t)} \\&+ \epsilon^{-1} \int_0^\infty   \int_{B(x,2t)} V v^2   \phi^2 \left(\frac{d(x,y)}{t}\right) \frac{t \mathrm{d}t \mathrm{d}y}{Vol(y,t)}\\
&\leq C \epsilon^{-1} \int_0^\infty   \int_{B(x,2t)} \left[|\nabla_{t,y} v|^2 + Vv^2\right] \frac{t \mathrm{d}t \mathrm{d}y}{Vol(y,t)}.
\end{align*}
The last inequality is obtained by choosing $\epsilon$ small enough. Now we deal with $I_1$. After integrations by parts (in $y$ and $t$) and using $(\frac{\partial^2}{\partial t^2} - \Delta-V) e^{-tL^{1/2}}f = 0$ we obtain
\begin{align*}
|I_1| &\leq \int_0^\infty \int_{M} |u-v| \left| \nabla_{t,y} u . \nabla_{t,y}  \left[   \frac{t \phi^2(d(x,y)/t)}{Vol(y,t)}\right] \right| \mathrm{d}t \mathrm{d}y
\end{align*}
The doubling property \eqref{chap5-doubling} and Lemma \ref{chap5-lemme-derivation} yield
\begin{equation}
\left|\nabla_{t,y}  \left[   \frac{t \phi^2(d(x,y)/t)}{Vol(y,t)}\right] \right| \leq C \frac{\phi(d(x,y)/t) \theta(d(x,y)/t)}{Vol(y,t)}
\end{equation} where $\theta(s) = \phi(s) + |\phi'(s)|$.
Hence, by Young inequality
\begin{align*}
I_1  &\leq   C \left[ \epsilon\int_0^\infty  \int_{B(x,2t)} |\nabla_{t,y} u|^2 \frac{t \mathrm{d}t \mathrm{d}y}{Vol(y,t)} + \epsilon^{-1} \int_0^\infty  \int_{B(x,2t)} |u-v|^2 \frac{\mathrm{d}t \mathrm{d}y}{t Vol(y,t)} \right]  \\
&\leq C \epsilon^{-1} \int_0^\infty  \int_{B(x,2t)} |u-v|^2 \frac{\mathrm{d}t \mathrm{d}y}{t Vol(y,t)}.
\end{align*}
The last inequality is obtained by choosing epsilon small enough.
\end{proof} 
As a consequence we can state the following theorem.
\begin{theorem}
Assume that $M$ satisfies the doubling property \eqref{chap5-doubling}, then $P_L$ is bounded on $L^p$ for $p \in [2,+\infty)$. 
\end{theorem}
\begin{proof}
Fix $p \in [2,+\infty)$. Lemma \ref{chap5-lemme 6} gives 
\begin{align*}\|P_L(f)\|_p &\leq C \left[ \|\G_L(f)\|_p + \left\| \left(  \int_0^\infty \int_{B(x,2t)} |\nabla_{t,y} e^{-t^2 L} f|^2 \frac{ t \mathrm{d}y \mathrm{d}t}{Vol(y,t)} \right)^{1/2} \right\|_p \right] \\ &\leq C \left[
\|f \|_p + \left\| \left( \int_0^\infty \left| \left( e^{-t L^{1/2}} - e^{-t^2 L} \right) f \right|^2 \frac{\mathrm{d}t}{t} \right)^{1/2}\right\|_p \right].
\end{align*}
The second part of the RHS term is the $L^p$ norm of the horizontal square function associated with $\phi(z) = e^{-z^{1/2}} - e^{-z}$, and is then bounded by $C \|f \|_p$.
\end{proof}


\section{Study of $\vec{P}$}
In this very short section, we introduce the conical square function associated with the Poisson semigroup on 1-forms. It is defined as follows.
\begin{equation*}
\vec{P}(\omega)(x)= \left( \int_0^\infty \int_{B(x,t)}   |d^* e^{-t\LF^{1/2}} \omega|^2 + |d e^{-t\LF^{1/2}} \omega|^2+ |\frac{\partial}{\partial t} e^{-t\vec{\Delta}^{1/2}} \omega |^2 \frac{t \mathrm{d}t \mathrm{d}y}{Vol(y,t)}\right)^{1/2} . 
\end{equation*}
We denote by $\vec{P}_t$ the time derivative part of $P$, $\vec{P}_d$ the derivative part and $\vec{P}_{d^*}$ the co-derivative part. We denote by $\vec{P}_x$ the part with both the derivative and the co-derivative.
\begin{align*}
\vec{P}_t(\omega)(x) &= \left( \int_0^\infty \int_{B(x,t)}  |\frac{\partial}{\partial t} e^{-t\LF^{1/2}} \omega|^2  \frac{t \mathrm{d}t \mathrm{d}y}{Vol(y,t)}\right)^{1/2}, \\
\vec{P}_x(\omega)(x) &= \left( \int_0^\infty \int_{B(x,t)}   |d^* e^{-t\LF^{1/2}} \omega|^2 + |d e^{-t\LF^{1/2}} \omega|^2  \frac{t \mathrm{d}t \mathrm{d}y}{Vol(y,t)}\right)^{1/2}. 
\end{align*}
We obtain as for $P_L$ the following result.
\begin{proposition}
$\vec{P}$ is bounded on $L^2$
\end{proposition}
The boundedness of these functionals may have consequences concerning the boundedness of the Riesz transform. We make some comments in the following sections.

\section{Lower bounds}
In this section, we prove that the boundedness of conical square functionals on $L^p$ implies lower bounds on the dual space $L^{p'}$.
\begin{theorem}\label{chap5-dual}
Let $F: \mathbb{R}_+ \mapsto \mathbb{C}$ be a function in $L^2(\mathbb{R_+})$ such that $F(0) \ne 0$. If $\G_L^F$ is bounded on $L^p$ then there exists $C >0$ such that for all $f \in L^{p'}$,
$$ \left\| f  \right\|_{p'} \leq C \left\| \G_L^F (f) \right\|_{p'}.$$
\end{theorem}
\begin{proof}
Let $f$ be in $L^p \cap L^2$ and $g$ be in $L^{p'} \cap L^2$. By integration by parts,
\begin{align*}
\int_0^\infty \int_M &\nabla F(tL) f . \overline{\nabla F(tL) g}  \mathrm{d}t \mathrm{d}x + \int_0^\infty \int_M\sqrt{V} F(tL) f . \overline{\sqrt{V} F(tL) g} \mathrm{d}t \mathrm{d}x \\ &= \int_0^\infty \int_M L  F(tL) f .\overline{  F(tL) g}  \mathrm{d}t \mathrm{d}x \\
&= \int_0^\infty \int_M L  |F(tL)|^2 f . \overline{g}  \mathrm{d}t \mathrm{d}x
\end{align*}
Set  $\mathcal{F}(\lambda) = \int_\lambda^\infty |F(t)|^2 \mathrm{d}t$. One has $\mathcal{F} ( \lambda)  \rightarrow 0 $ when $\lambda \rightarrow + \infty$. Therefore, the spectral resolution gives $\mathcal{F}(tL) f \rightarrow 0$ as $t \rightarrow + \infty$. The spectral resolution also implies that $\frac{\partial}{\partial t  } \mathcal{F}(tL)^2 = - L|F|^2(tL).$ From this we obtain
\begin{align*}
\int_0^\infty &\int_M L |F|^2(tL)f . \overline{g} \mathrm{d}t \mathrm{d}x  \\
&= \int_0^\infty \int_M -\frac{\partial}{\partial t } \mathcal{F}(tL) f . \overline{g} \mathrm{d}t \mathrm{d}x \\
&=  \int_M f . \overline{\mathcal{F}(0) g} \mathrm{d}x.
\end{align*}
 Using all the forgoing equalities and the same averaging trick as in the former proofs,
 \begin{align*}
 &\left| \int_M f.\overline{\mathcal{
 F}(0)g} \mathrm{d}x \right| \\
 &=\int_0^\infty \int_M \nabla F(tL) f . \overline{\nabla F(tL) g} + \sqrt{V} F(tL) f . \overline{\sqrt{V} F(tL) g  } \mathrm{d}t \mathrm{d}x \\
&= \int_0^\infty \int_M \int_{B(x,t^{1/2})} \nabla F(tL) f . \overline{\nabla F(tL) g } \frac{\mathrm{d}t \mathrm{d}x \mathrm{d}y }{Vol(x,t^{1/2})} \\ &+ \int_0^\infty \int_M \int_{B(x,t^{1/2})}  \sqrt{V} F(tL) f . \overline{\sqrt{V} F(tL) g } \frac{\mathrm{d}t \mathrm{d}x \mathrm{d}y }{Vol(x,t^{1/2})} \\
&= \int_0^\infty \int_M \int_{B(y,t^{1/2})} \nabla F(tL) f . \overline{\nabla F(tL) g} \frac{\mathrm{d}t \mathrm{d}x \mathrm{d}y }{Vol(x,t^{1/2})} \\ &+ \int_0^\infty \int_M \int_{B(y,t^{1/2})}  \sqrt{V} F(tL) f . \overline{\sqrt{V} F(tL) g} \frac{\mathrm{d}t \mathrm{d}x \mathrm{d}y }{Vol(x,t^{1/2})}.
  \end{align*}
  The Cauchy-Schwarz (in $t$) and Hölder (in $y$) inequalities give
   \begin{align*}
 &\left| \int_M f.\overline{\mathcal{F}(0)g} \mathrm{d}x \right| \\
 &\leq \int_M \left[ \int_0^\infty \int_{B(y,t^{1/2})}  |\nabla F(tL) f|^2 + V |F(tL)f |^2 \frac{\mathrm{d}t \mathrm{d}x}{Vol(x,t^{1/2}) } \right]^{1/2} \\ &\times  \left[ \int_0^\infty \int_{B(y,t^{1/2})} |\nabla F(tL) g|^2 + V |F(tL) g |^2  \frac{\mathrm{d}t \mathrm{d}x}{Vol(x,t^{1/2}) } \mathrm{d}y \right]^{1/2}  \\
 &\leq \| \G_L^F(f) \|_p \|\G_L^F(g)\|_{p'} \\
 &\leq C \| f\|_p \|\G_L^F(g)\|_{p'}.
 \end{align*}
 We obtain the result by taking the supremum on $f$ in the unit ball of $L^p(M)$.
\end{proof}

One can also state a result about lower bounds concerning the functionals associated with the Poisson semigroup. They are not included in the latter theorem because of the time derivative part.

\begin{proposition} If $P_L$ is bounded on $L^p$, then the reverse inequality $$\| f \|_{p'} \leq C \| P_L(f) \|_{p'} $$ holds for all $f \in L^{p'}$.
\end{proposition}
\begin{proof}
Fix $f$ in $L^p \cap L^2$ and $g$ in $L^{p'} \cap L^2$. By integration by parts,
\begin{align*}
 \int_M f(x) g(x) \mathrm{d}x &= \int_0^\infty \frac{\partial}{\partial t}\int_M e^{-tL^{1/2}} f . e^{-tL^{1/2}} g \mathrm{d}t \mathrm{d}x \\
&=\int_0^\infty t \frac{\partial^2}{\partial t^2} \int_M  e^{-tL^{1/2}} f .  e^{-tL^{1/2}} g \mathrm{d}x \mathrm{d}t \\
&= \int_0^\infty 2 t \int_M \left( L^{1/2} e^{-tL^{1/2}} f .  L^{1/2} e^{-tL^{1/2}} g \right) \mathrm{d}x \mathrm{d}t \\
&+  \int_0^\infty 2 t \int_M \left(   L e^{-tL^{1/2}} f . e^{-tL^{1/2}} g  \right) \mathrm{d}x \mathrm{d}t  \\
&= 2 \int_0^\infty  \int_M \left( t\nabla_x e^{-tL^{1/2}} f . t\nabla_x e^{-tL^{1/2}} g \right) \frac{\mathrm{d}x \mathrm{d}t}{t} 
\\&+   2 \int_0^\infty  \int_M \left(  tV^{1/2} e^{-tL^{1/2}} f .  tV^{1/2} e^{-tL^{1/2}} g \right) \frac{\mathrm{d}x \mathrm{d}t}{t} \\
&+  \int_0^\infty  \int_M \left(  t\frac{\partial}{\partial t} e^{-tL^{1/2}} f .  t\frac{\partial}{\partial t} e^{-tL^{1/2}} g \right) \frac{\mathrm{d}x \mathrm{d}t}{t}
\end{align*}
By Cauchy-Schwarz inequality (in $t$) and the same averaging trick as for $\G_L$ we obtain 
\begin{equation*}
\left| \int_M f(x) g(x) \mathrm{d}x \right| \leq C \|P_L(f)\|_p \|P_L(g)\|_{p'}.
\end{equation*}
The boundedness $P_L$ on $L^p$ and taking the supremum on $f$  gives $\|g \|_{p'} \leq C \| P_L(g) \|_{p'}$.
\end{proof} 
\begin{remark}The same result still holds if we only consider $P_{L,x}$ or $P_{L,t}$.
\end{remark}

We obtain the same result for $\vec{P}$.

\begin{proposition} If $\vec{P}$ is bounded on $L^p$, then the reverse inequality $$\| \omega \|_{p'} \leq C \| \vec{P}(\omega) \|_{p'} $$ holds for all $\omega \in L^{p'}$. The result remains true if we consider only $\vec{P}_x$ or  $\vec{P}_t$.
\end{proposition}
\section{Link with the Riesz transform}\label{chap5-section-riesz}
Some links between Littlewood-Paley-Stein functions and the Riesz transforms have been established in \cite{Coulhon-Duong-2003}. We make analogous links between conical square functions and the Riesz transform. They rely on Theorem \ref{chap5-dual} together with the commutation formula $d \Delta = \vec{\Delta} d$.
\begin{theorem}\label{chap5-theo10-1}
\begin{enumerate}
\item If $P_{\Delta,x}$ is bounded on $L^p$ and $\vec{P_t}$ is bounded on $L^{p'}$ then the Riesz transform is bounded on $L^p$.
\item If $\vec{P_{x}}$ is bounded on $L^p$ and  $P_{\Delta,t}$ is bounded on $L^{p'}$ then the Riesz transform is bounded on $L^{p'}$.
\end{enumerate}
\end{theorem}
\begin{proof}
We prove the first item. The second is proven by duality considering that $d^* \LF^{-1/2}$ is the adjoint of $d \Delta^{-1/2}$.
If the  $\vec{P}_t$ is bounded on $L^{p'}$, then by the reverse inequality on $L^p$ one has
\begin{align*}
\|d f\|_p &\leq C \|  \vec{P}_t (df) \|_p\\
&=C \left\|\left( \int_0^\infty \int_{B(x,t^{1/2})} |\vec{\Delta}^{1/2} e^{-t\vec{\Delta}^{1/2}} df|^2 \frac{\mathrm{d}y \mathrm{d}t}{Vol(y,t^{1/2})}    \right)^{1/2} \right\|_p\\
&= C \left\|\left( \int_0^\infty \int_{B(x,t^{1/2})} |d e^{-t\Delta^{1/2}} \Delta^{1/2} f|^2  \frac{\mathrm{d}y \mathrm{d}t}{Vol(y,t^{1/2})}     \right)^{1/2} \right\|_p\\
&= C \left\| P_x (\Delta^{1/2} f) \right\|_p \\
&\leq C \left\| \Delta^{1/2} f \right\|_p.
\end{align*}
For the second equality we used commutation formula $d \Delta = \vec{\Delta} d$. For the last inequality we used of the boundedness of $P_x$ on $L^p$.
\end{proof}

\begin{remark}
\begin{enumerate}
\item Fix $p\in[2,+\infty)$. Assuming \eqref{chap5-doubling}, $P_x$ is bounded on $L^p$. Then the boundedness of $\vec{P}_t$ on $L^{p'}$ implies the boundedness of Riesz transform on $L^p$. Unfortunately, for $p \leq 2$, $\vec{P}_t$ is even harder to bound than the horizontal Littlewood-Paley-Stein function for $\vec{\Delta}$  (which is known to be difficult for all $p \in (1, \infty)$). This can be done under subcriticality assumption on the negative part of the Ricci via Stein's method but we only recover a known result about Riesz transform.
\item For $p\in [2,+\infty)$, $\vec{\G}$ is bounded on $L^p$ if we assume  the \eqref{chap5-doubling}. Using a similar proof as in Theorem \ref{chap5-theo10-1} we see that it is sufficient to bound the functional 
$$\SSS_{\phi_0} (f) (x) = \left( \int_0^\infty \int_{B(x,t^{1/2})} | \Delta^{1/2} e^{-t\Delta} f|^2 \frac{\mathrm{d}y \mathrm{d}y}{Vol(y,t)}\right)^{1/2} $$ on $L^{p'}$ to obtain the boundedness of the Riesz transform on $L^{p'}$.
\end{enumerate}
\end{remark}
We recover a result from \cite{Chen-Magniez-Ouhabaz}, that is the boundedness of the Riesz transform under the hypothesis of Theorem \ref{chap5-ricci-subcritical-case}. The functional $\vec{\G}$ is bounded on $L^p$ for $p \in (p_0,2)$ by Theorem \ref{chap5-ricci-subcritical-case}. The functional $\SSS_{\phi_0}$ satisfies the reverse inequality for $p$ in this range, so the adjoint of the Riesz transform $d^* \vec{\Delta}^{-1/2}$ is bounded. It implies the boundedness of $d \Delta^{-1/2}$ on $L^p$ for $p \in [2,{p'}_0)$. More generally, it gives a proof of the following theorem.

\begin{theorem}
Let $p$ be in $(1,2]$. Suppose that $M$ satisfies the doubling property \eqref{chap5-doubling} and that $\sqrt{t}d^* e^{-t\vec{\Delta}}$ satisfies $L^p-L^2$ estimates \eqref{chap5-LpL2gradient}, then the Riesz transform is bounded on $L^{p'}$.
\end{theorem}

\bibliographystyle{plain}
\bibliography{bibli}

\end{document}